\documentclass[12pt]{amsart}
\usepackage[osf,sc]{mathpazo}
\usepackage{amssymb}

\usepackage{geometry}
\usepackage{bm}
\usepackage{graphicx}
\usepackage{hyperref}
\hypersetup{
    colorlinks=true, 
    linktoc=all,     
    linkcolor=blue,
        citecolor=red,
    filecolor=black,
    urlcolor=blue	
}
\usepackage{pdfpages}
\usepackage{enumerate}
\usepackage[inline]{enumitem}
\makeatletter
\newcommand{\inlineitem}[1][]{%
\ifnum\enit@type=\tw@
    {\descriptionlabel{#1}}
  \hspace{\labelsep}%
\else
  \ifnum\enit@type=\z@
       \refstepcounter{\@listctr}\fi
    \quad\@itemlabel\hspace{\labelsep}%
\fi} \makeatother
\newcommand{\gth}{\theta}

\newcommand{\gp}{\pi}

\newcommand{\gs}{\sigma}

\newcommand{\gf}{\phi}

\newcommand{\Gd}{\Delta}

\newcommand{\Gom}{\Omega}

\newcommand{\bs}{\backslash}

\newcommand{\nin}{\notin}

\newcommand{\ti}{\tilde}

\newcommand{\mbb}{\mathbb}

\newcommand{\mcl}{\mathcal}

\newcommand{\lra}{\longrightarrow}

\newcommand{\Ra}{\Rightarrow}
\newcommand{\Llra}{\Longleftrightarrow}

\newcommand{\equ}[1]{%
\begin{equation*}
#1
\end{equation*}
}
\newcommand{\equa}[1]{%
\begin{equation*}
\begin{aligned}
#1
\end{aligned}
\end{equation*}
}

\newtheorem{theorem}{Theorem}[section]

\newtheorem{prop}[theorem]{Proposition}
\newtheorem{lemma}[theorem]{Lemma}

\newtheorem{ques}[theorem]{Question}

\newtheorem{example}[theorem]{Example}
\newtheorem{note}[theorem]{Note}

\makeatletter
\def\namedlabel#1#2{\begingroup
   \def\@currentlabel{#2}%
   \label{#1}\endgroup
}
\makeatother
\newtheorem*{thmA}{\bf{Theorem A}}
\newtheorem*{thmB}{\bf{Theorem B}}

\theoremstyle{definition}
\newtheorem{defn}[theorem]{Definition}
\theoremstyle{remark}

\numberwithin{equation}{section}
\begin{document}
\title[On the Triangles in Certain Types of Line Arrangements]{On the Triangles in Certain Types of Line Arrangements}
\author[C.P. Anil Kumar]{Author: C.P. Anil Kumar*}
\address{Post Doctoral Fellow in Mathematics, Harish-Chandra Research Institute,
Chhatnag Road, Jhunsi, Prayagraj (Allahabad)-211019, Uttar Pradesh, INDIA
}
\email{akcp1728@gmail.com}
\thanks{*The work is done when the author is a Post Doctoral Fellow at HRI, Allahabad.}
\begin{abstract}
In this article, we combinatorially describe the triangles that are present in two types of line arrangements, those which have global cyclicity and those which are infinity type line arrangements. A combinatorial nomenclature has been described for both the types and some properties of the nomenclature have been proved. Later using the nomenclature we describe the triangles present in both types of line arrangements in main Theorems~\ref{theorem:TLAGC},~\ref{theorem:TITLA}. We also prove that the set of triangles uniquely determine, in a certain precise sense, the line arrangements with global cyclicity and not the infinity type line arrangements where counter examples have been provided. In Theorem~\ref{theorem:LineAtInfinity}, given a nomenclature, we characterize when a particular line symbol in the nomenclature is a line at infinity for the arrangement determined by the nomenclature. 
\end{abstract}
\maketitle
\section{\bf{Introduction}}
Line arrangements in the plane have been studied extensively in the literature in various contexts (B.~Grunbaum~\cite{MR0307027} and the references therein). The authors such as H.~Harborth~\cite{MR0809189}, G.~B.~Purdy~\cite{MR0523090},~\cite{MR0566442},~\cite{MR0560588}, J.~P.~Roudneff~\cite{MR0852111},~\cite{MR0892402},~\cite{MR0920696},~\cite{MR1368516}, D.~Ljubic, J.~P.~Roudneff, B.~Sturmfels~\cite{MR0978064}, Z.~Furedi, I.Palasti~\cite{MR0760946}, G.~J~Simmons~\cite{MR0295191} and T.~O.~Strommer~\cite{MR0462985} have worked on different aspects of triangles, quadrilaterals and pentagons present in line arrangements either in the Euclidean plane or in the projective plane.  However a combinatorial characterization of the triangles in a line arrangement in the Euclidean plane has not been done before. Here for certain types of line arrangements we characterize the triangles present combinatorially and mention some consequences. This characterization requires a certain combinatorial nomenclature for line arrangements. This is done for two types of line arrangements, those which have global cyclicity (Definition~\ref{defn:GlobalCyclicity}) and those which are of infinity type (Definition~\ref{defn:InfinityTypeLineArrangement}). In the last section, we prove an important theorem of characterizing certain types of lines using the nomenclature of infinity type arrangements.
\section{\bf{Definitions}}
In this section we mention a few definitions.
\begin{defn}[Lines in Generic Position in the Plane $\mbb{R}^2$ or Line Arrangement]
	\label{defn:linearrangement}
	~\\
	Let $n$ be a positive integer. We say that a finite set $\mcl{L}_n=\{L_1,L_2,\ldots,L_n\}$ of lines in $\mbb{R}^2$ is in a generic position or is a line arrangement if the following two conditions hold.
	\begin{enumerate}
		\item No two lines are parallel.
		\item No three lines are concurrent.
	\end{enumerate}
	In this case we say that $\mcl{L}_n$ is a line arrangement. We say $n$ is the cardinality of the line arrangement.
\end{defn}

Now we give the definition of an isomorphism between two line arrangements. 
\begin{defn}[Isomorphism]
	\label{defn:Iso}
	~\\
	Let $n,m$ be positive integers. Let \equ{\mcl{L}^1_n=\{L^1_1,L^1_2,\ldots,L^1_n\},\mcl{L}^2_m=\{L^2_1,L^2_2,\ldots,L^2_m\}} 
	be two line arrangements in the plane $\mbb{R}^2$ of cardinalities $n,m$ respectively. 
	We say that a map $\gf:\mcl{L}^1_n \lra \mcl{L}^2_m$ is an isomorphism between the line arrangements if 
	\begin{enumerate}
		\item the map $\gf$ is a bijection, (that is, $n=m$) with $\gf(L^1_i)=L^2_{\gf(i)},1\leq i\leq n$ and
		\item for any $1\leq i\leq n$ the order of intersection vertices on the lines $L^1_i,L^2_{\gf(i)}$ agree via the bijection induced by $\gf$ on its subscripts. There are four possibilities of pairs of orders and any one pairing of orders out of the four pairs must agree via the bijection induced by $\gf$ on its subscripts. 
	\end{enumerate}
	Two mutually opposite orders of points arise on any line in the plane. We say that the isomorphism $\gf$ is trivial on subscripts if in addition we have $\gf(i)=i,1\leq i\leq n$.
\end{defn}
\begin{note}
Henceforth we shall assume that $\mcl{L}_n=\{L_1,L_2,\ldots,L_n\}$ is a line arrangement in the plane with respective angles $0<\gth_1<\gth_2<\ldots<\gth_n<\gp$ as a convention where angles are made with respect to the positive $X\operatorname{-}$axis.
\end{note}

Here we define an equivalence relation on the set of triangles that are present in any line arrangement.
\begin{defn}[An Equivalence Relation]
	\label{defn:EquivalenceTriangles}
	~\\
	Let $\mcl{L}_n=\{L_1,L_2,\ldots,L_n\}$ be a line arrangement in the plane. We say that two triangles 
	$\Gd L_{a}L_{b}L_{c},\Gd L_dL_eL_f, 1\leq a<b<c\leq n,1\leq d<e<f\leq n$ in the line arrangement $\mcl{L}_n$ are corner adjacent if $\{a,b,c\}\cap\{d,e,f\}$ has precisely two elements. We say that two triangles $\Gd L_{a}L_{b}L_{c},\Gd L_dL_eL_f, 1\leq a<b<c\leq n,1\leq d<e<f\leq n$ are equivalent if they are same or if there is a sequence of triangles $\Gd L_{a_i}L_{b_i}L_{c_i},1\leq a_i<b_i<c_i\leq m,1\leq i\leq n$ in the line arrangement $\mcl{L}_n$ with $a_1=a,b_1=b,c_1=c,a_m=d,b_m=e,c_m=f$ such that $\Gd L_{a_i}L_{b_i}L_{c_i}$ is corner adjacent to $\Gd L_{a_{i+1}}L_{b_{i+1}}L_{c_{i+1}}$ for $1\leq i\leq m-1$. It is clear that this is an equivalence relation. 
\end{defn}
\begin{example}
	\begin{figure}[h]
		\centering
		\includegraphics[width = 0.8\textwidth]{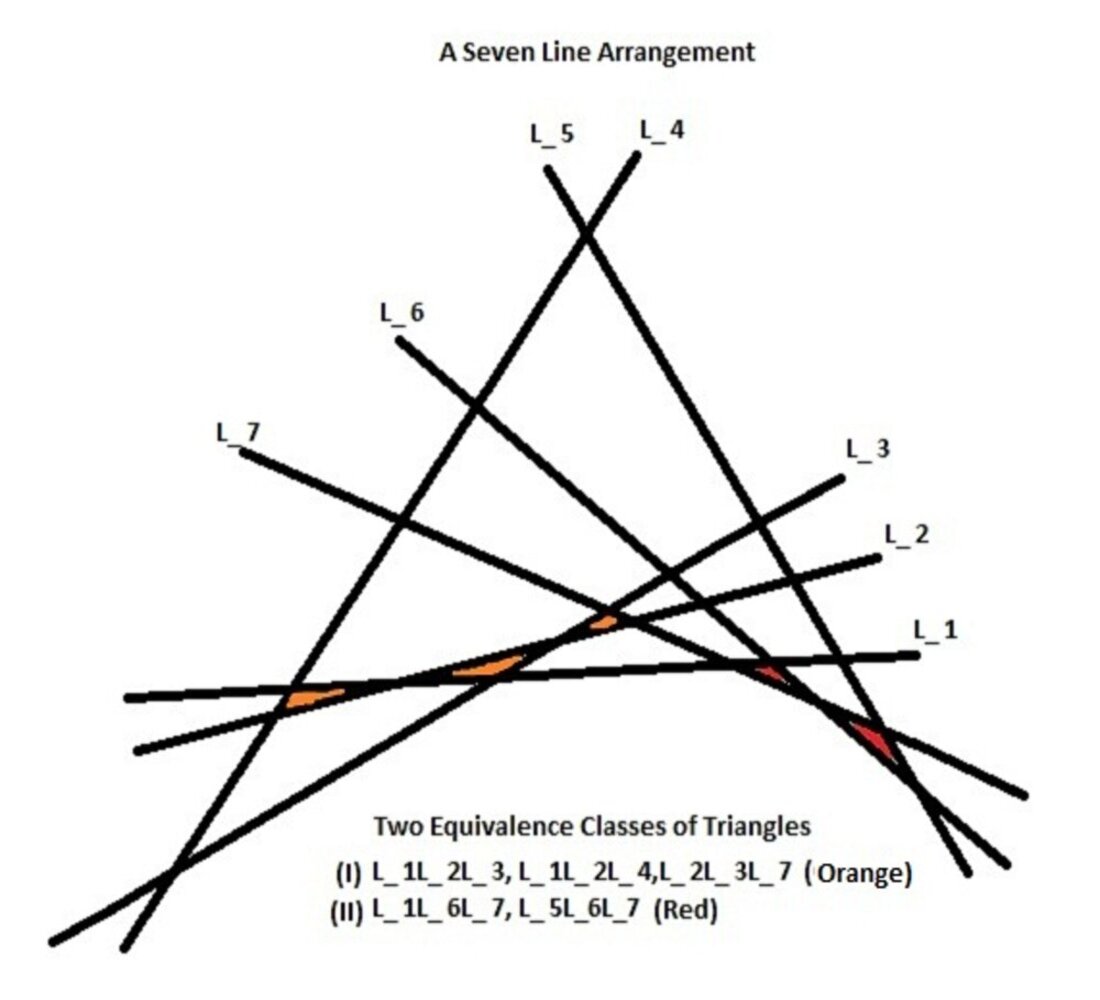}
		\caption{Seven-Line Arrangement}
		\label{fig:Zero}
	\end{figure}
We give an example of a seven-line arrangement with two equivalence classes in Figure~\emph{\ref{fig:Zero}}. One equivalence class of triangles contains $\Gd L_1L_2L_4,\Gd L_1L_2L_3$, $\Gd L_2L_3L_7$. Another equivalence class of triangles contains $\Gd L_1L_6L_7,\Gd L_5L_6L_7$.
The concept of adjacent triangles in a line arrangement is also mentioned in D.~Ljubic, J.~P.~Roudneff, B.~Sturmfels~\cite{MR0978064}.
\end{example}

We define the corner points of a line arrangement in the plane.
\begin{defn}
Let $\mcl{L}_n=\{L_1,L_2,\ldots,L_n\}$ be a line arrangement in the plane. For any $1\leq i\neq j\leq n$ the vertex $P=L_i\cap L_j$ is a said to be a corner point if the point $P$ is the end point of intersection on both the lines $L_i$ and $L_j$. In Figure~\ref{fig:Zero} the corner vertices are $L_3\cap L_4,L_4\cap L_5,L_5\cap L_6$.
\end{defn}

\section{\bf{Definitions, Nomenclature and Main Theorem on Line Arrangements with Global Cyclicity}}
In this section we mention the required definitions to state the first main Theorem~\ref{theorem:TLAGC}.
Now we define a line arrangement with global cyclicity.
\begin{defn}[Nomenclature: Existence of Global Cyclicity with an $n\operatorname{-}$cycle]
\label{defn:GlobalCyclicity}

Let $\mcl{L}_n=\{L_1,L_2,\ldots,L_n\}$ be a line arrangement in the plane. We say that there exists global cyclicity in the line arrangement $\mcl{L}_n$ if all the lines form the sides of a convex $n\operatorname{-}$gon in some cyclic order of the lines. 
Suppose the angle made by the line $L_i$ with respect to the positive $X\operatorname{-}$axis is $\gth_i,1\leq i\leq n$ and suppose we have $0< \gth_1< \gth_2<\ldots <\gth_n<\gp$ (conventional notation). Let the convex $n\operatorname{-}$gon in the anticlockwise cyclic order be given by 
\equ{L_{1=a_1}\lra L_{a_2}\lra \ldots \lra L_{a_{n-1}}\lra L_{a_n}\lra L_{a_1}.} Then we say that the line arrangement $\mcl{L}_n$ has global cyclicity having gonality $n\operatorname{-}$cycle $(1=a_1a_2\ldots a_n)$. For this type of line arrangement the nomenclature is just the cycle $(1=a_1a_2\ldots a_n)$. This cycle has the property that there exists $2\leq r\leq n-1$ such that 
\begin{itemize}
\item $1=a_1<a_2<\ldots<a_r$,
\item $a_{r+1}<a_{r+2}<\ldots<a_n$,
\item $1<a_{r+1}<a_r$.
\end{itemize}
Moreover any such cycle can occur as a nomenclature of a conventional line arrangement with global cyclicity.
\end{defn}
\subsection{\bf{The First Main Theorem}}
We state the first main theorem of the article.
\begin{thmA}[Triangles Determine the Line Arrangement with Global Cyclicity]
\namedlabel{theorem:TLAGC}{A}
~\\
Let $n\geq 4$ be a positive integer and $\mcl{L}^i_n=\{L^i_1,L^i_2,\ldots,L^i_n\},i=1,2$ be two line arrangements in the plane. Let the angle made by the line $L^i_j$ with respect to the positive $X\operatorname{-}$axis be $\gth^i_j$ and suppose $0<\gth^i_1<\gth^i_2<\ldots<\gth^i_n<\gp$ for $i=1,2$. Suppose in addition both the line arrangements have global cyclicity with anti-clockwise gonality $n\operatorname{-}$cycles $(1=a_1^ia_2^i\ldots a_n^i),i=1,2$ respectively. Then the following $(1)\operatorname{-}(3)$ are equivalent.
\begin{enumerate}[label=\emph{(\arabic*)}]
\item The line arrangements $\mcl{L}^i_n,i=1,2$ are isomorphic with the isomorphism which is trivial on subscripts.
\item The gonality cycles $(1=a_1^ia_2^i\ldots a_n^i),i=1,2$ are equal or same.
\item $L^1_iL^1_jL^1_k$ is a triangle in line arrangement $\mcl{L}^1_n$ if and only if $L^2_iL^2_jL^2_k$ is a triangle in line arrangement $\mcl{L}^2_n$ for any $1\leq i<j<k\leq n$.
\end{enumerate}
There are at most two equivalence classes of triangles. 
Moreover if $1=a_1^i<a_2^i<\ldots<a_r^i;\ a_{r+1}^i<a_{r+2}^i<\ldots<a_n^i;\ 1=a_1^i<a^i_{r+1}<a^i_r,\ i=1,2$ for some $2\leq r\leq n-1$ then the triangles are given by
\begin{itemize}
\item $\Gd L^i_{a^i_{j}}L^i_{a^i_{j+1}}L^i_{a^i_{j+2}},1\leq j<j+2\leq r$ or $r+1\leq j<j+2\leq n$,
\item $\Gd L^i_{a^i_1}L^i_{a^i_{n-1}}L^i_{a^i_n}$ if $n\geq r+2$,
\item $\Gd L^i_{a^i_{1}}L^i_{a^i_{2}}L^i_{a^i_{n}}$ if $a^i_{2}<a^i_{n}$,
\item $\Gd L^i_{a^i_{r+1}}L^i_{a^i_{r-1}}L^i_{a^i_{r}}$ if $a^i_{r+1}<a^i_{r-1}$,
\item $\Gd L^i_{a^i_{r+1}}L^i_{a^i_{r+2}}L^i_{a^i_{r}}$ if $a^i_{r+2}<a^i_{r}$ and $n\geq r+2$.
\end{itemize}
There are $2^{n-1}-n$ such isomorphism classes.
\end{thmA}

\section{\bf{Definitions, Nomenclature and Main Theorem on Infinity Type Line Arrangements}}
In this section we mention the required definitions to state the second main theorem.
Here we define a line at infinity for a line arrangement.
\begin{defn}[Line at Infinity with respect to a line arrangement]
	\label{defn:LineatInfinity}
	~\\
	Let $\mcl{L}_n=\{L_1,L_2,\ldots,L_n\}$ be a line arrangement in the plane. We say that a line $L$ is a line at infinity with respect to $\mcl{L}_n$ if $\mcl{L}_n \cup \{L\}$ is a line arrangement and all the vertices, that is, zero dimensional intersections of the lines of the arrangement $\mcl{L}_n$ lie on ``one side" of $L$ (possibly the ``one side" includes the line $L$ if $L\in \mcl{L}_n$).
\end{defn}
Now we define an infinity type line arrangement.
\begin{defn}[Infinity Type Line Arrangement]
	\label{defn:InfinityTypeLineArrangement}
	~\\
	Let $\mcl{L}_n=\{L_1,L_2,\ldots,L_n\}$ be a line arrangement in the plane.
	We say that $\mcl{L}_n$ is an infinity type line arrangement if there exists a permutation $\gs\in S_n$ such that the line $L_{\gs(l)}$ is a line at infinity with respect to the arrangement 
	\equ{\{L_{\gs(1)},L_{\gs(2)},\ldots,L_{\gs(l-1)}\}, 2\leq l\leq n.}
	The permutation $\gs$ is said to be an infinity permutation of the line arrangement $\mcl{L}_n$. It need not be unique.
\end{defn}
\subsection{\bf{Nomenclature for an Infinity Type Line Arrangement}}
In this section we define a nomenclature to describe an infinity type line arrangement.
\subsubsection{\bf{Conventions and Fixing the Orientation of Lines}}
Let $L$ be a line in the plane. Let $\gth$ be the angle with respect to the positive $X\operatorname{-}$axis
with $0<\gth<\gp$. We translate the line and conventionally assume that the line meets the positive $X\operatorname{-}$axis. Now the line $L$ meets 
\begin{enumerate}
\item either the three quadrants IV,I,II when $\frac{\gp}{2}<\gth<\gp$,
\item or the three quadrants III,IV,I when $0<\gth<\frac{\gp}{2}$,
\item or only two quadrants IV,I when the line $L$ is parallel to $Y\operatorname{-}$axis.
\end{enumerate}
We give conventional orientation to the line $L$ in various cases $1,2,3$ according to the increasing $y\operatorname{-}$co-ordinate values on the line as
\begin{enumerate}
\item IV$\lra$ I$\lra$ II,
\item III$\lra$ IV$\lra$ I,
\item IV$\lra$ I.
 \end{enumerate}
If $\mcl{L}=\{L_1,L_2,\ldots,L_n\}$ is a line arrangement in the plane then by a suitable translation we assume that all the vertices of intersections lie in the first quadrant and all the lines of the arrangement intersect the positive $X\operatorname{-}$axis and they are all conventionally oriented.
\subsubsection{\bf{Nomenclature for a Triangle}}
\label{subsubsec:NT}
Consider three conventional oriented lines $L_i,L_j,L_k$ in the plane with $i,j,k\in \mbb{N}$ with respective angles $0<\gth_i<\gth_j<\gth_k<\gp,i<j<k$ where the angles are made with respect to the positive $X\operatorname{-}$axis and the vertices $L_i\cap L_j,L_j\cap L_k,L_i\cap L_k$ lie in the first quadrant. There are two possibilities as shown in Figure~\ref{fig:One}.
\begin{figure}[h]
	\centering
	\includegraphics[width = 0.8\textwidth]{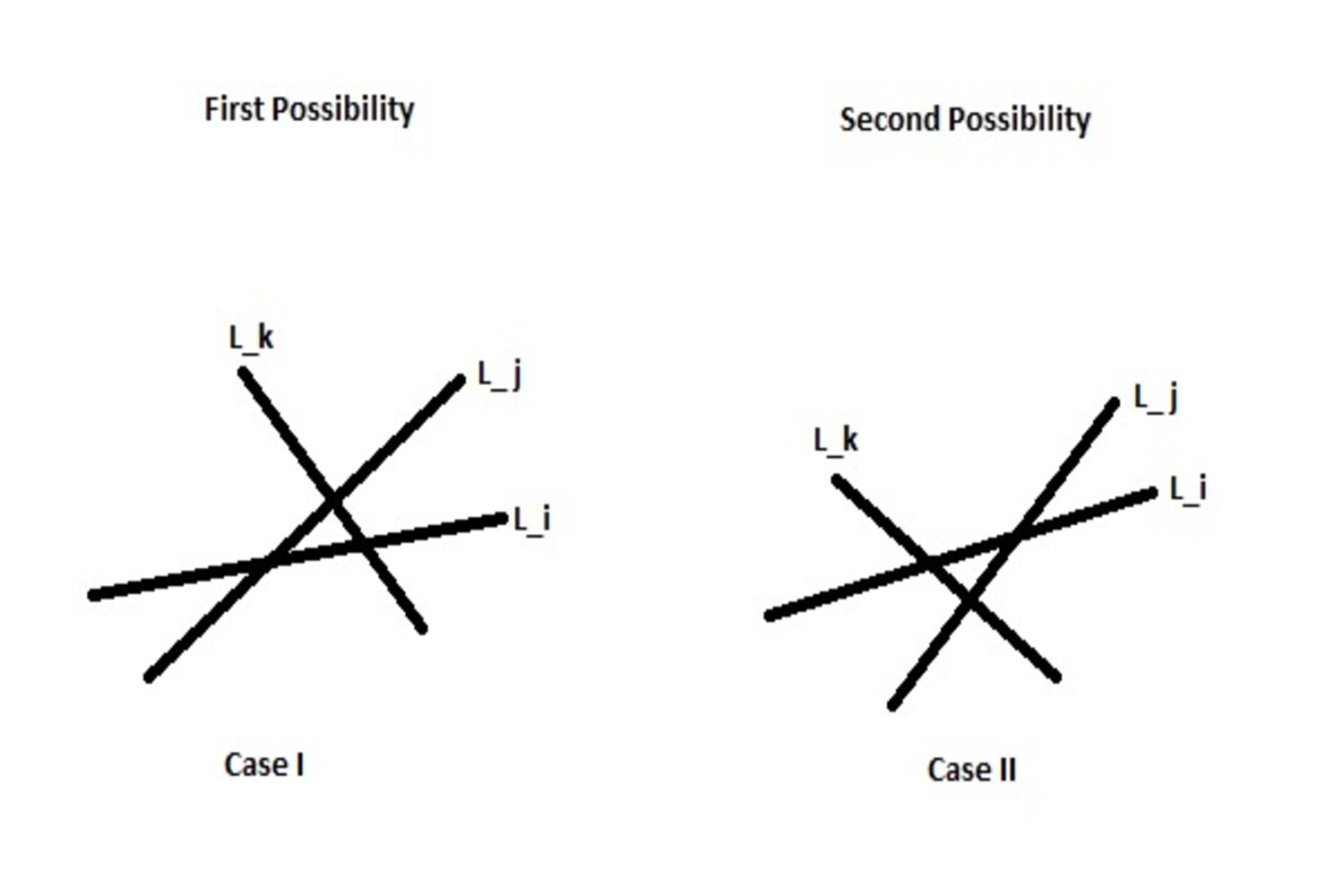}
	\caption{Two Possibilities for a Triangle $\Gd L_iL_jL_k$}
	\label{fig:One}
\end{figure}
The nomenclature here consists of the symbols $i,j,k$ in some order with superscripts for each, a $``+1"$ sign or a $``-1"$ sign. We give a $``+1"$ sign superscript to $k$ if the order of intersections on the oriented line $L_k$ is first $i$ then $j$ (as in Case I). Otherwise we give a $``-1"$ sign superscript (as in Case II).
We give a $``+1"$ sign superscript to $i$ if the order of intersections on the oriented line $L_i$ is first $j$ then $k$ (as in Case I). Otherwise we give a $``-1"$ sign superscript (as in Case II). We give a $``-1"$ sign superscript to $j$ if the order of intersections on the oriented line $L_j$ is first $i$ then $k$ (as in Case I). Otherwise we give a $``+1"$ sign superscript (as in Case II).
The following nomenclatures describe the triangle $\Gd L_iL_jL_K$ in Case I.
\equ{i^{+1}j^{-1}k^{+1},i^{+1}k^{+1}j^{-1},j^{-1}i^{+1}k^{+1},j^{-1}k^{+1}i^{+1},k^{+1}i^{+1}j^{-1},k^{+1}j^{-1}i^{+1}}
The following nomenclatures describe the triangle $\Gd L_iL_jL_K$ in Case II.
\equ{i^{-1}j^{+1}k^{-1},i^{-1}k^{-1}j^{+1},j^{+1}i^{-1}k^{-1},j^{+1}k^{-1}i^{-1},k^{-1}i^{-1}j^{+1},k^{-1}j^{+1}i^{-1}}
If, in addition, we fix the order of $i,j,k$ in any of the above then we have a unique nomenclature in both cases.

Now an equivalent criterion for the assignment of superscripts is given as follows.
We observe that if the line $L_k$ does not separate the origin and the vertex $L_i\cap L_j$ in two different half planes then a $``+1"$ superscript is attached and if it does then a $``-1"$ superscript is attached. Similarly for the line $L_j$, if it does not separate the origin and the vertex $L_i\cap L_k$ in two different half planes then a $``+1"$ superscript is attached and if it does then a $``-1"$ superscript is attached. Also similarly if the line $L_i$ does not separate the origin and the vertex $L_j\cap L_k$ in two different half planes then a $``+1"$ superscript is attached and if it does then a $``-1"$ superscript is attached.
With this equivalent criterion for the assignment of superscripts we give a nomenclature for an infinity type line arrangement.

\subsubsection{\bf{Nomenclature for an Infinity Type Line Arrangement}}
~\\
Let $\mcl{L}_n=\{L_1,L_2,\ldots,L_n\}$ be a conventional infinity type line arrangement in the plane with an infinity permutation $\gp$. A nomenclature is made up of certain symbols as follows. The order of the lines of the arrangement is given by $\gp(1)\gp(2)\gp(3)\ldots \gp(n)$. We associate the superscripts $``+1,-1"$ as follows. The lines $L_{\gp(1)},L_{\gp(2)},L_{\gp(3)}$ form a triangle of the arrangement as the permutation $\gp$ is an infinity permutation of the infinity type line arrangement $\mcl{L}_n$. So we use the nomenclature of the triangle in Section~\ref{subsubsec:NT} and assign superscripts $``+1,-1"$ signs to $\gp(1),\gp(2),\gp(3)$. Now for $l\geq 3$, if $L_{\gp(l)}$ does not separate the origin on one side and the vertices of intersections of the lines of the line arrangement $\{L_{\gp(1)},L_{\gp(2)},\ldots,L_{\gp(l-1)}\}$ on the other side then a $``+1"$ superscript is attached to $\gp(l)$. If it does separate then a $``-1"$ superscript is attached to $\gp(l)$. 

\begin{example}
Consider the seven line arrangements in Figure~\emph{\ref{fig:Zero}}. This line arrangement has a nomenclature of symbols as $1^{+1}2^{-1}3^{+1}7^{+1}6^{+1}4^{-1}5^{+1}$. It also has a nomenclature as $1^{+1}2^{-1}3^{+1}4^{-1}7^{+1}6^{+1}5^{+1}$.
\end{example}

\begin{note}[About Uniqueness of the Nomenclature]
~\\
As we have seen in general the nomenclature is not unique for an infinity type line arrangement though one such nomenclature always exists. The nomenclature is unique for a given infinity permutation. Even otherwise the nomenclature is unique in the following sense. We define uniquely an infinity permutation $\gs$ for $\mcl{L}_n$ as follows. Since there always exists a line at infinity for $\mcl{L}_n$ let $L_{\gs(n)}$ be the one with largest subscript $\gs(n)$. Then the following $(n-1)$-line arrangement $\{L_1,L_2,\ldots,L_n\}\bs\{L_{\gs(n)}\}$ is also an infinity type line arrangement. Now we pick a line $L_{\gs(n-1)}$ at infinity of the $(n-1)$-line arrangement having largest subscript $\gs(n-1)$ with $L_{\gs(n-1)}\in \{L_1,L_2,\ldots,L_n\}\bs\{L_{\gs(n)}\}$. Inductively we continue this process to define the infinity permutation $\gs$ uniquely and hence the nomenclature in this manner is uniquely obtained.
\end{note}
\subsection{\bf{The Second Main Theorem}}
Now we state the second main theorem of the article.
\begin{thmB}[Triangles of an Infinity Type Line Arrangement]
\namedlabel{theorem:TITLA}{B}
~\\
Let $\mcl{L}_n=\{L_1,L_2,\ldots,L_n\}$ be an infinity type line arrangement in the plane with an infinity permutation $\gp$ and nomenclature $\gp(1)^{a_1}\gp(2)^{a_2}\gp(3)^{a_3}\ldots \gp(n)^{a_n}$ where $a_i\in \{+1,-1\},1\leq i\leq n$. For $1\leq i<j<k\leq n$ the lines $L_{\gp(i)},L_{\gp(j)},L_{\gp(k)}$ form a triangle then there is a necessary condition to be satisfied which is as follows. 
\begin{itemize}
\item Necessary Condition: 
\begin{enumerate}[label=\emph{(\arabic*)}]
\item Either there are no integers in the set $\{\gp(1),\gp(2),\ldots,\gp(k)\}$
which are in between $\gp(i)$ and $\gp(j)$ 
\item or all integers in the set $\{\gp(1),\gp(2),\ldots,\gp(k)\}$
lie between $\gp(i)$ and $\gp(j)$ (including the end values).
\end{enumerate}
\end{itemize}
In addition to the necessary condition the lines $L_{\gp(i)},L_{\gp(j)},L_{\gp(k)}$ form a triangle of the line arrangement $\mcl{L}_n$ if and only if 
\begin{itemize}
\item $i=1,j=2,k=3$ or else,
\item  $(1)$ occurs and we should have 
\begin{enumerate}[label=\emph{(\alph*)}]
\item $a_k=sign\big(a_j(\gp(j)-\gp(i))(\gp(k)-\gp(j)\big)$
\item and for any $k>l>j, a_l=-sign\big(a_j(\gp(j)-\gp(i))(\gp(l)-\gp(j)\big)$.
\end{enumerate}
\item or else $(2)$ occurs and we should have $a_k=a_j$ and for any $k>l>j, a_l=-sign(a_j)$.
\end{itemize}
\equ{\boxed{Nomenclature: \ldots \gp(i)^{a_i}\ldots \gp(j)^{a_j}\ldots \gp(l)^{a_l} \ldots \gp(k)^{a_k}\ldots }}
\end{thmB}
\section{\bf{Proof of the First Main Theorem}}
In this section we first prove three Propositions~[\ref{prop:TriinLineArrGlobalCyclicity},\ref{prop:NomenclatureIsomorphism},\ref{prop:TriangleProperty}] which are required to prove the first main Theorem~\ref{theorem:TLAGC} later.

\begin{prop}
\label{prop:TriinLineArrGlobalCyclicity}
Let $\mcl{L}_n=\{L_1,\ldots,L_n\}$ be a conventional line arrangement in the plane with global cyclicity.
Then all the triangles that occur in the line arrangement are edge-adjacent to the convex $n\operatorname{-}$gon.
\end{prop}
\begin{proof}
	\begin{figure}[h]
	\centering
	\includegraphics[width = 0.8\textwidth]{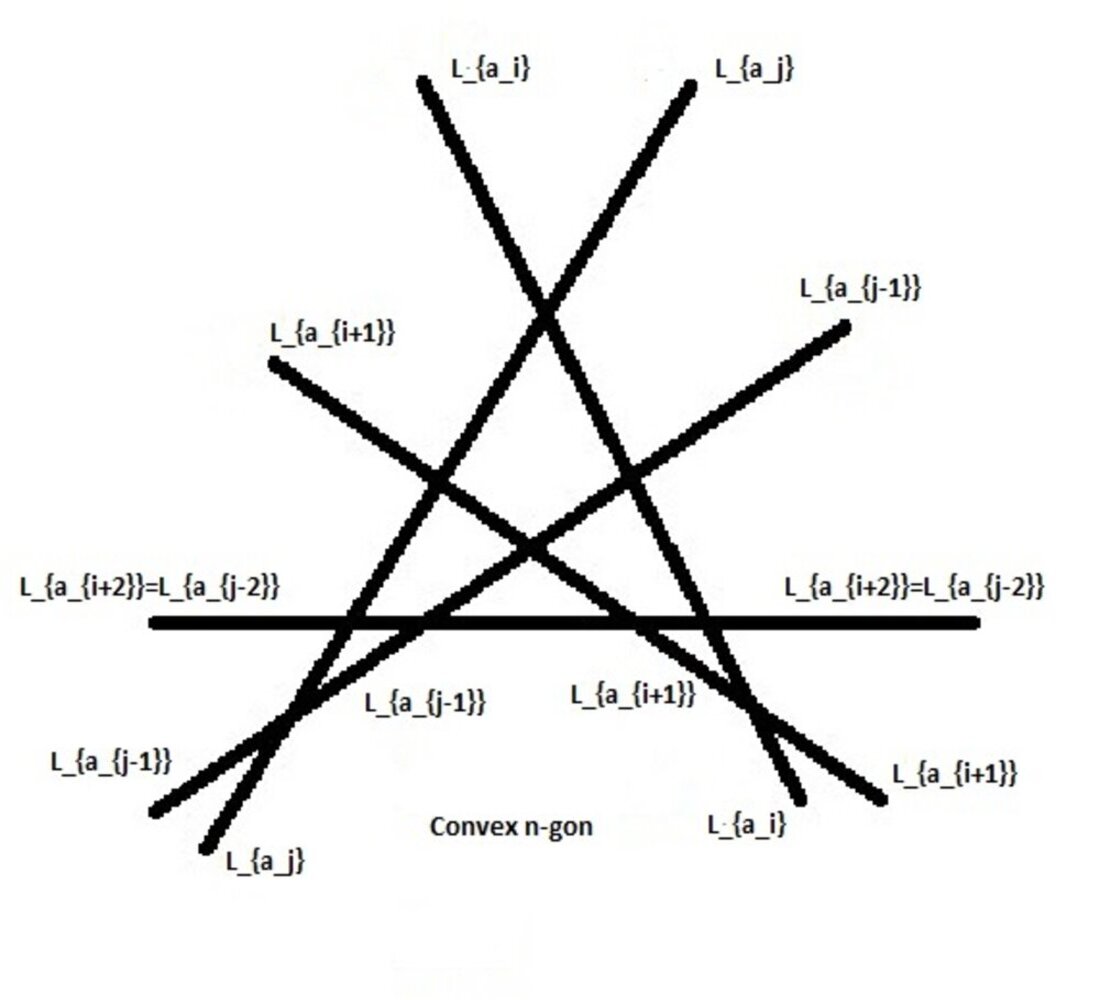}
	\caption{Illustration of Triangles Edge-Adjacent to the Convex $n\operatorname{-}$gon}
	\label{fig:Two}
\end{figure}
Let $(1=a_1a_2\ldots a_n)$ be the gonality cycle in the usual anti-clockwise order. Hence the convex $n\operatorname{-}$gon is given by $L_{a_1}\lra \ldots \lra L_{a_n}\lra L_{a_1}$ in the anti-clockwise manner. Let $P=L_{a_i}\cap L_{a_j},1\leq i<j\leq n$ be a corner vertex. Consider all the regions bounded by $L_{a_i},L_{a_j}$ and some of the edges of the convex $n\operatorname{-}$gon which are either 
\equ{L_{a_{i+1}},L_{a_{i+2}},\ldots, L_{a_{j-1}}} 
or 
\equ{L_{a_{j+1}},L_{a_{j+2}},\ldots, L_{a_n},L_{a_1},\ldots, L_{a_{i-1}}} 
depending on which side $L_{a_i}$ and $L_{a_j}$ meet. Now we observe that in this, the regions are either quadrilaterals or triangles and the triangles occur edge-adjacent to the convex $n\operatorname{-}$gon. For example consider Figure~\ref{fig:Two} for illustration.
Now all the bounded regions must occur in this manner for some corner vertex. Hence the bounded regions apart from the convex $n\operatorname{-}$gon are either quadrilaterals or triangles. This also proves the proposition that all the triangles that occur in the line arrangement $\mcl{L}_n$ are edge-adjacent to the convex $n\operatorname{-}$gon.
\end{proof}
\begin{prop}
\label{prop:NomenclatureIsomorphism}
Let $\mcl{L}^i_n=\{L^i_1,\ldots,L^i_n\},i=1,2$ be two conventional line arrangements in the plane with global cyclicity having gonality $n\operatorname{-}$cycles $\gs^i=(1=a^i_1a^i_2\ldots a^i_n)$, $i=1,2$ respectively. Then the bijection $\gf:\mcl{L}^1_n \lra \mcl{L}^2_n$, $\gf(L^1_j)=L^2_j,1\leq j\leq n$ is an isomorphism if and only if $\gs^1=\gs^2$. 
\end{prop}
\begin{proof}
Suppose $\gf$ is an isomorphism then the combinatorial data of all the respective convex polygons in the arrangements agree and hence $\gs_1=\gs_2$. Now conversely if $\gs_1=\gs_2$ and further there is $2\leq r \leq n-1$ such that $1=a_1^i<\ldots<a_r^i;a_{r+1}^i<\ldots<a_n^i;1\leq a^i_{r+1}<a_r^i$ then we have the same order of intersections on the line $L^i_{a_1}=L^i_1,i=1,2$ and it is given by 
\equ{L^i_{a_{r+1}}\cap L^i_1 \lra \ldots \lra L^i_{a_n}\cap L^i_1 \lra L^i_{a_2}\cap L^i_1 \lra L^i_{a_r}\cap L^i_1.}
To obtain the order of intersections on the line $L_2$ we do the following. We cyclically renumber all the lines so that $L_2$ becomes $L_1$. Now we recover in a similar manner the order of intersections on the newly renumbered line $L_1$. This is because we get a similar gonality cycle with a new value of $r$. Then we revert back to old numbering to obtain the order of intersections on the line $L_2$. This way we continue till $L_n$ to obtain combinatorially the same order of intersections for $i=1,2$ on any two respective lines $L^i_j,i=1,2$ for $1\leq j\leq n$.
This shows that $\gf$ is an isomorphism and completes the proof of the proposition.
\end{proof}
\begin{prop}
\label{prop:TriangleProperty}
Let $n\geq 4$ be a positive integer and let $\mcl{L}_n=\{L_1,\ldots,L_n\}$ be a conventional line arrangement in the plane with global cyclicity. Let $L_a\lra L_b\lra L_c,1\leq a,b,c\leq n$ be three anti-clockwise juxtaposed sides of the convex $n\operatorname{-}$gon with $L_b$ the middle side. Then the three lines $L_a,L_b,L_c$ form a triangle of the arrangement if and only if $a>c>b$ or $b>a>c$ or $c>b>a$.
\end{prop}
\begin{proof}
In a conventional line arrangement for any three lines $L_i,L_j,L_k$ with $1\leq i<j<k\leq n$, the orientation in general of the triangle $\Gd L_iL_jL_k$ is always clockwise with the orientation as given by $L_i\lra L_j\lra L_k\lra L_i$. Note that $\Gd L_iL_jL_k$ in general need not be a triangular region of the arrangement as it can have subdivisions into smaller regions. Now if the lines $L_a,L_b,L_c$ form a triangle of the arrangement then it is edge-adjacent via the middle edge $L_b$ to the convex $n\operatorname{-}$gon. We observe that if $L_a\lra L_b\lra L_c$ is anti-clockwise juxtaposed for the convex $n\operatorname{-}$gon then it forms a triangle  of the arrangement if and only if $a>c>b$ or $b>a>c$ or $c>b>a$. This proves the proposition.  
\end{proof}
Now we prove the first main Theorem~\ref{theorem:TLAGC}.
\begin{proof}
$(1)\Llra (2)$ follows from Proposition~\ref{prop:NomenclatureIsomorphism}. So $(1)$ or $(2)$ implies $(3)$ is immediate. Moreover from $(2)$, using Proposition~\ref{prop:TriangleProperty}, we can list the triangles in the isomorphic arrangements as given in the theorem. These are the only triangles of the isomorphic arrangements using Proposition~\ref{prop:TriinLineArrGlobalCyclicity}. It is also clear if a line arrangement has global cyclicity then there are at most two equivalence classes of triangles in the arrangement. Now we prove that $(3)$ implies $(2)$. Here the sets of triangles that arise from line arrangements with global cyclicity are same with the same combinatorial descriptions. We can read off the following three strings of inequalities for $i=1,2$ in a unique manner from the combinatorial description of triangles.
\begin{itemize}
	\item $1=a_1^i<a_2^i<\ldots<a_r^i$,
	\item $a_{r+1}^i<a_{r+2}^i<\ldots<a_n^i$,
	\item $1=a_1^i<a^i_{r+1}<a^i_{r}$.
\end{itemize}
From these we can obtain the same cycles $(1=a_1^ia_2^i\ldots a_n^i),i=1,2$. This proves $(3)\Ra (2)$. By counting we find that there are $2^{n-1}-n$ such gonality cycles. This completes the proof of the theorem. 
\end{proof}
\section{\bf{Corner Lemma and Triangle Lemma}}
In this section we prove two basic guiding lemmas which are very useful in the proof of the second main Theorem~\ref{theorem:TITLA}.

\begin{lemma}[Corner Lemma]
\label{lemma:Corner}
~\\
Consider the axes and two other lines $L,M$ giving rise to a four line arrangement in the plane.  Then the origin is a corner point of the four line arrangement if and only if $L$ and $M$, each meets the same set of three quadrants.
As a consequence the respective angles $\gth_L,\gth_M$ of the lines $L$ and $M$ make with respect to the positive $X\operatorname{-}$axis, both lie in $(0,\frac{\gp}{2})$ or both lie in $(\frac{\gp}{2},\gp)$. 
\end{lemma}
\begin{proof}
Since $L,M$ are not parallel to either of the axes, they meet three quadrants. The three quadrants can be any one of the following.
\begin{itemize}
	\item IV,I,II.
	\item I,II,III.
	\item II,III,IV.
	\item III,IV,I.
\end{itemize} 
Now it is clear that the origin is the corner point if and only if $L$ and $M$ meets the same set of three quadrants. The assertion about the angles is also clear.
\end{proof}
Now we prove another important lemma.
\begin{lemma}[Triangle Lemma]
\label{lemma:TriangleLemma}
~\\
Let $n\in \mbb{N}$. Consider the axes and finitely many lines $L_1,L_2,\ldots,L_n$ giving rise to a line arrangement consisting of $(n+2)$ lines.
\begin{enumerate}[label=\emph{(\arabic*)}]
\item Then the origin is a corner point if and only if each of the lines $L_i,1\leq i\leq n$ meets the same three quadrants. As a consequence all the angles of these lines with respect to the positive $X\operatorname{-}$axis lie in either the angle interval $(0,\frac{\gp}{2})$ or the angle interval $(\frac{\gp}{2},\gp)$.
\item Now suppose the origin is a corner point and let $L$ be a new line at infinity to the arrangement. 
\begin{enumerate}[label=\emph{(\alph*)}]
\item The axes and the line $L$ form a triangle if and only if the angle of $L$ lies in the same angle interval as that of lines $L_i, 1\leq i\leq n$ and the line $L$ does not meet the same set of three quadrants which all the lines $L_i,1\leq i\leq n$ meet.
\item In this scenario if we orient all the lines $L,L_i, 1\leq i\leq n$ according to the increasing $y\operatorname{-}$co-ordinate value then the line $L$ and the lines $L_i,1\leq i\leq n$ have opposite orders of intersections with axes.
\end{enumerate}
\end{enumerate}
\end{lemma}	
\begin{proof}
$(1)$ follows from Lemma~\ref{lemma:Corner}. The proof of $2(a)$ and $2(b)$ is also straight forward.
\end{proof}
\section{\bf{Proof of the Second Main Theorem}}
In this section we prove the second main Theorem~\ref{theorem:TITLA}.
\begin{proof}
Let $\gth_i$ be the angle made by the line $L_i$ with respect to the positive $X\operatorname{-}$axis for $1\leq i\leq n$. So we have $0<\gth_1<\gth_2<\ldots<\gth_n<\gp$. We prove the forward implication.
For $1\leq i<j<k\leq n$ the lines $L_{\gp(i)},L_{\gp(j)},L_{\gp(k)}$ form a triangle then $L_{\gp(i)}\cap L_{\gp(j)}$ is a corner vertex for the line arrangement $\{L_{\gp(1)},L_{\gp(2)},\ldots,$ $L_{\gp(k-1)}\}$. Now using Triangle Lemma~\ref{lemma:TriangleLemma} applied for the corner vertex $L_{\gp(i)}\cap L_{\gp(j)}$, we conclude that the angles $\gth_{\gp(t)}$ of the line $L_{\gp(t)}$ for $1\leq t\leq k$, all lie in between $\gth_{\gp(i)}$ and $\gth_{\gp(j)}$ or all lie in between $\max(\gth_{\gp(i)},\gth_{\gp(j)})$ and $\big(\gp+\min(\gth_{\gp(i)},\gth_{\gp(j)})\big)$ where are angles are considered congruent modulo $\gp$.  This gives the necessary condition that either there are no integers in the set $\{\gp(1),\gp(2),\ldots,\gp(k)\}$ which are in between $\gp(i)$ and $\gp(j)$ or all integers in the set $\{\gp(1),\gp(2),\ldots,\gp(k)\}$ lie between $\gp(i)$ and $\gp(j)$ (including the end values).

If $i=1,j=2,k=3$ then the theorem holds true. So assume that this is not the case. Hence there exists $t_0$ such that  $1\leq t_0<k,i\neq t_0\neq j$. If $i=1,j=2$ then we choose $t_0=3$. If $j>2$ then we choose $t_0<j,t_0\neq i$, say $t_0=1$ if $i\neq 1$ and $t_0=2$ if $i=1$. 

Consider Figure~\ref{fig:Three} where the quadrants are depicted in all cases. All the lines are oriented in the direction of increasing $y\operatorname{-}$co-ordinate values.
\begin{figure}[h]
\centering
\includegraphics[width = 1.0\textwidth]{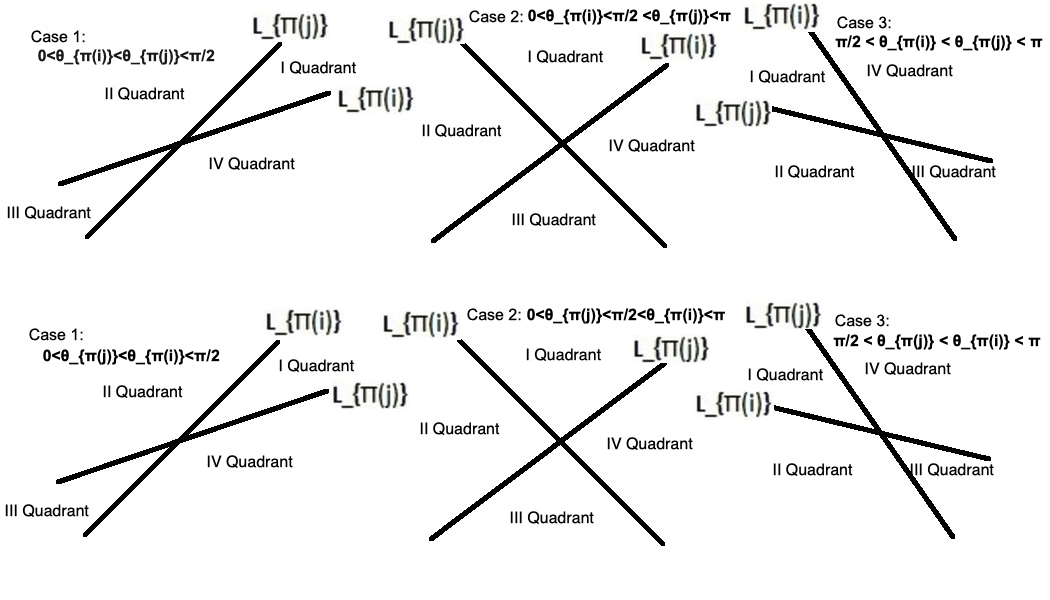}
\caption{Depiction of Quadrants with respect to Corner vertex $L_{\gp(i)}\cap L_{\gp(j)}$ where $L_{\gp(i)}, L_{\gp(j)}$ make respective angles $\gth_{\gp(i)},\gth_{\gp(j)}$ with respect to the positive X-axis, when $\gp(i)<\gp(j)$ cases 1,2,3 and when $\gp(j)<\gp(i)$ cases 1,2,3}
\label{fig:Three}
\end{figure}
We assume first that $\gp(t_0)$ does not lie in between $\gp(i)$ and $\gp(j)$ so that necessary condition $(1)$ occurs. The line $L_{\gp(t_0)}$ is used later in the proof as a reference line to obtain combinatorial data.
Now we observe the following given in a table.
\begin{center}
\begin{tabular}{ |c|c|c|c| } 
\hline	 Inequality & $a_j$ & Quads of $L_{\gp(l)}$, & $a_k$, Quads of $L_{\gp(k)}$ \\
& & $j<l<k,a_l$ & \\
\hline	 $\gp(i)<\gp(j)$ & +1 &  II,III,IV if & IV,I,II and\\
& & $\gp(l) \nin [\gp(i),\gp(j)]$. & $\gp(k) \nin [\gp(i),\gp(j)]$.\\
& & $\gp(l)<\gp(i)\Ra a_l=+1$ & $\gp(k)<\gp(i)\Ra a_k=-1$\\ 
& & $\gp(l)>\gp(j)\Ra a_l=-1$ & $\gp(k)>\gp(j)\Ra a_k=+1$\\
\hline	$\gp(i)<\gp(j)$ & -1 &  IV,I,II if & II,III,IV and\\  
& & $\gp(l) \nin [\gp(i),\gp(j)]$. & $\gp(k) \nin [\gp(i),\gp(j)]$.\\
& & $\gp(l)<\gp(i)\Ra a_l=-1$ & $\gp(k)<\gp(i)\Ra a_k=+1$\\ 
& & $\gp(l)>\gp(j)\Ra a_l=+1$ & $\gp(k)>\gp(j)\Ra a_k=-1$\\
\hline	$\gp(i)>\gp(j)$ & +1 & IV,I,II if & II,III,IV and\\ 
& & $\gp(l) \nin [\gp(j),\gp(i)]$. & $\gp(k) \nin [\gp(j),\gp(i)]$.\\
& & $\gp(l)<\gp(j)\Ra a_l=-1$ & $\gp(k)<\gp(j)\Ra a_k=+1$\\ 
& & $\gp(l)>\gp(i)\Ra a_l=+1$ & $\gp(k)>\gp(i)\Ra a_k=-1$\\
\hline	$\gp(i)>\gp(j)$ & -1 & II,III,IV if & IV,I,II and\\ 
& & $\gp(l) \nin [\gp(j),\gp(i)]$. & $\gp(k) \nin [\gp(j),\gp(i)]$.\\
& & $\gp(l)<\gp(j)\Ra a_l=+1$ & $\gp(k)<\gp(j)\Ra a_k=-1$\\ 
& & $\gp(l)>\gp(i)\Ra a_l=-1$ & $\gp(k)>\gp(i)\Ra a_k=+1$\\
\hline
\end{tabular}
\end{center} 
We mention the proof of one row of the above table. We consider only the case $\gp(i)<\gp(j),a_j=+1,\gp(l)>\gp(j),\gp(k)>\gp(j)$. The proof for the rest of the cases is similar. 

Now $a_j=+1$ implies that the line $L_{\gp(j)}$ does not separate the origin and the point $L_{\gp(t_0)}\cap L_{\gp(i)}$. Since $\gp(t_0)$ does not lie in between $\gp(i)$ and $\gp(j), L_{\gp(t_0)}$ meets the quadrants II,III,IV. Now $L_{\gp(i)}\cap L_{\gp(j)}$ is a corner point for the arrangement $\{L_{\gp(1)},L_{\gp(2)},\ldots,L_{\gp(k-1)}\}$. Hence we have for any $j<l<k$ the line $L_{\gp(l)}$ meets the same set of quadrants which $L_{\gp(t_0)}$ meets which is II,III,IV using Lemma~\ref{lemma:TriangleLemma}(1). So we have, if $\gp(l)>\gp(j)$ then the line $L_{\gp(l)}$ separates the origin and $L_{\gp(i)}\cap L_{\gp(j)}$. This implies that $a_l=-1$. The lines $L_{\gp(i)},L_{\gp(j)},L_{\gp(k)}$ form a triangle implies that $\gp(k)$ does not lie in between $\gp(i)$ and $\gp(j)$ and $L_{\gp(k)}$ has to meet the quadrants IV,I,II and does not meet III. Now if $\gp(k)>\gp(l)$ then the line $L_{\gp(k)}$ does not separate origin and $L_{\gp(i)}\cap L_{\gp(j)}$. This implies $a_k=+1$.

From the above table, we have proved that if the necessary condition $(1)$ occurs then we should have $a_k=sign\big(a_j(\gp(j)-\gp(i))(\gp(k)-\gp(j)\big)$ and for any $k>l>j, a_l=-sign\big(a_j$ $(\gp(j)-\gp(i))(\gp(l)-\gp(j)\big)$, it is just the exact opposite. The proof of the converse is also similar if $(1)$ holds as each step is reversible. 

Now if the necessary condition $(2)$ occurs then $\gp(t_0)$ lies in between $\gp(i)$ and $\gp(j)$.
The line $L_{\gp(t_0)}$ is again used later in the proof as a reference line to obtain combinatorial data.
Now we observe the following given in a table.
\begin{center}
\begin{tabular}{ |c|c|c|c| } 
\hline	 Inequality & $a_j$ & Quads of $L_{\gp(l)}$, & $a_k$, Quads of $L_{\gp(k)}$ \\
& & $j<l<k,a_l$ & \\
\hline	 $\gp(i)<\gp(j)$ & +1 &  I,II,III if & III,IV,I and\\
& & $\gp(l) \in [\gp(i),\gp(j)]$. & $\gp(k) \in [\gp(i),\gp(j)]$.\\
& & $a_l=-1$ & $a_k=+1$.\\
\hline	 $\gp(i)<\gp(j)$ & -1 &  III,IV,I if & I,II,III and\\
& & $\gp(l) \in [\gp(i),\gp(j)]$. & $\gp(k) \in [\gp(i),\gp(j)]$.\\
& & $a_l=1$ & $a_k=-1$.\\
\hline	 $\gp(i)>\gp(j)$ & +1 &  I,II,III if & III,IV,I and\\
& & $\gp(l) \in [\gp(j),\gp(i)]$. & $\gp(k) \in [\gp(j),\gp(i)]$.\\
& & $a_l=-1$ & $a_k=+1$.\\
\hline	 $\gp(i)>\gp(j)$ & -1 &  III,IV,I if & I,II,III and\\
& & $\gp(l) \in [\gp(j),\gp(i)]$. & $\gp(k) \in [\gp(j),\gp(i)]$.\\
& & $a_l=1$ & $a_k=-1$.\\
\hline
\end{tabular}
\end{center} 
We mention the proof of one row of the above table. We consider only the case $\gp(i)<\gp(j),a_j=+1$. The proof for the rest of the cases is similar. 

Now $a_j=+1$ implies that the line $L_{\gp(j)}$ does not separate the origin and the point $L_{\gp(t_0)}\cap L_{\gp(i)}$. Since $\gp(t_0)$ lies in between $\gp(i)$ and $\gp(j)$, $L_{\gp(t_0)}$ meets the quadrants I,II,III. Now $L_{\gp(i)}\cap L_{\gp(j)}$ is a corner point for the arrangement $\{L_{\gp(1)},L_{\gp(2)},\ldots,L_{\gp(k-1)}\}$. Hence we have for any $j<l<k$ the line $L_{\gp(l)}$ meets the same set of quadrants which $L_{\gp(t_0)}$ meets which is I,II,III using Lemma~\ref{lemma:TriangleLemma}(1). This implies that $a_l=-1$. The lines $L_{\gp(i)},L_{\gp(j)},L_{\gp(k)}$ form a triangle implies that $\gp(k)$ lies in between $\gp(i)$ and $\gp(j)$ and $L_{\gp(k)}$ has to meet the quadrants III,IV,I and does not meet II. So the line $L_{\gp(k)}$ does not separate the origin and $L_{\gp(i)}\cap L_{\gp(j)}$. This implies $a_k=+1$.

From the above table, we have proved that if the necessary condition $(2)$ occurs then we should have $a_k=a_j$ and for any $k>l>j, a_l=-a_j$, it is just the exact opposite. The proof of the converse is also similar if $(2)$ holds as each step is reversible.

This completes the proof of the second main theorem.
\end{proof}
\begin{example}
Using Theorem~\emph{\ref{theorem:TITLA}} we can conclude that the triangles in the seven line arrangement in Fig~\emph{\ref{fig:Zero}} with the nomenclature $1^{+1}2^{-1}3^{+1}7^{+1}6^{+1}4^{-1}5^{+1}$ is exactly the set 
$\{\{1,2,3\},\{1,2,4\},\{2,3,7\},\{1,6,7\},\{5,6,7\}\}$. Using Theorem~\emph{\ref{theorem:TITLA}}, in fact, we can write a function in a computer which outputs the set of all triangles in an infinity type line arrangement whose nomenclature is given as an input to the function. The mathematica code for this, is given at the end of the article in Section~\ref{sec:Mathematica}.
\end{example} 

\section{\bf{Examples of Two Infinity Type Arrangements with the Same Set of Triangles}}
We mention some counter examples where the sets of triangles in two line arrangements are same but the line arrangements are not isomorphic. The precise statement is as follows.
\begin{example}
Let $\mcl{L}^i_n=\{L^i_1,L^i_2,\ldots,L^i_n\},i=1,2$ be  two line arrangements with respective angles $0<\gth^i_1<\ldots<\gth^i_n<\gp$ for the lines $L^i_j,1\leq j\leq n,i=1,2$ respectively where the angles are made with respect to the positive $X\operatorname{-}$axis. Let $\mcl{T}^i=\{\{j,k,l\}\mid 1\leq j<k<l\leq n,L^i_j,L^i_k,L^i_l\text{ form a triangle in $\mcl{L}_n^i$ }\}$, $i=1,2$. Let $\gf:\mcl{L}_n^1\lra \mcl{L}_n^2$ be the bijection such that $\gf(L^1_j)=L^2_j,1\leq j\leq n$. Then we have 
\begin{enumerate}[label=\emph{(\arabic*)}]
\item If $\gf$ is an isomorphism then $\mcl{T}^1=\mcl{T}^2$.
\item The converse need not hold. If $\mcl{T}^1=\mcl{T}^2$ then $\gf$ need not be an isomorphism.
\end{enumerate} 
\end{example}
It is clear that $(1)$ holds and it is easy to verify that for $3\leq n\leq 5$ the converse also holds. However for $n\geq 6$ the converse is not true. The counter examples are given in Figure~\ref{fig:Four}.  
\begin{figure}[h]
	\centering
	\includegraphics[width = 1.0\textwidth]{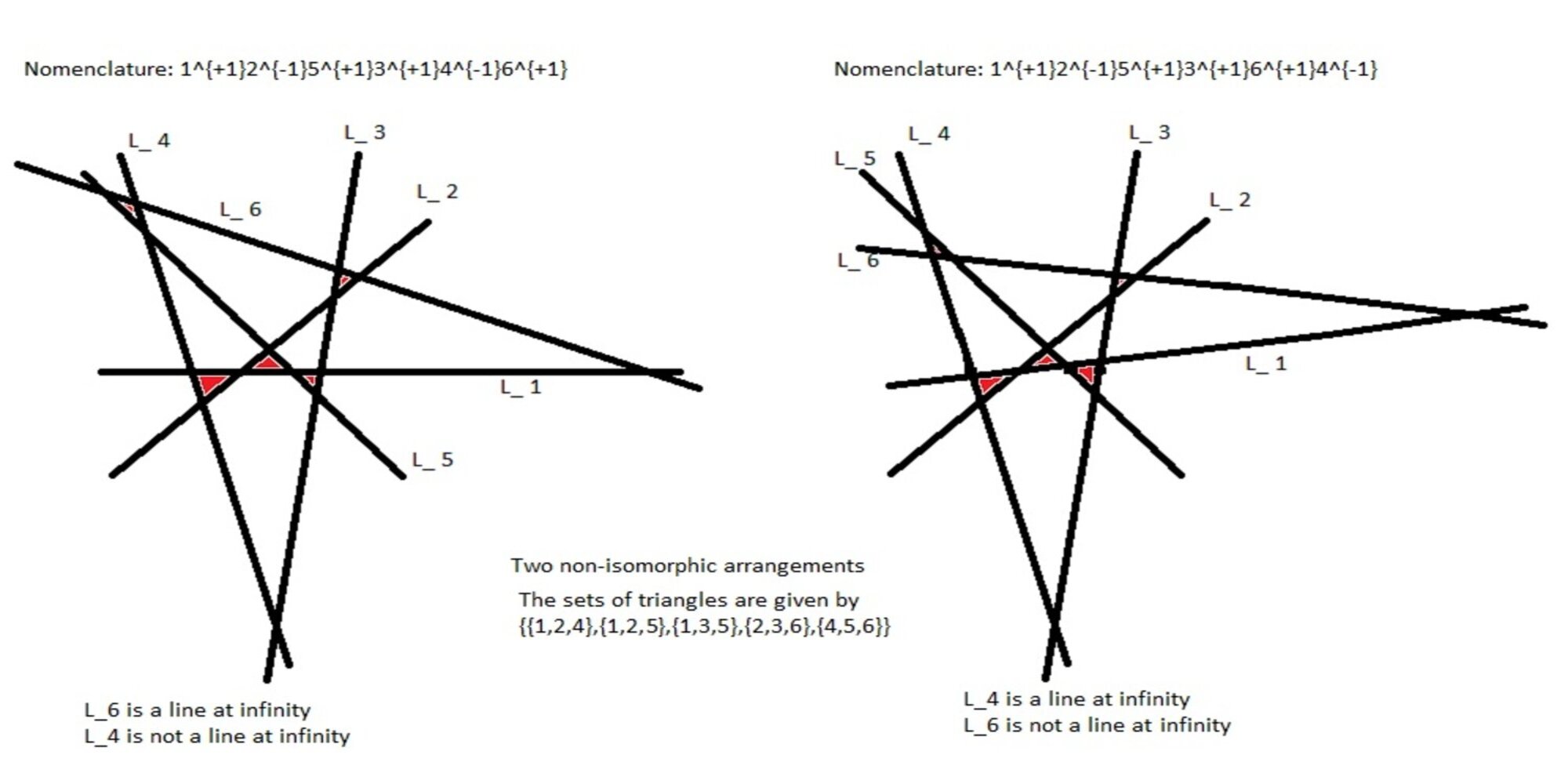}
	\caption{Two Six-Line Arrangements with the Same Set of Triangles}
	\label{fig:Four}
\end{figure}
Their nomenclatures are given by $1^{+1}2^{-1}5^{+1}3^{+1}4^{-1}6^{+1}$ and $1^{+1}2^{-1}5^{+1}3^{+1}6^{+1}4^{-1}$. They have the same sets of triangles, given by $\mcl{T}^1=\mcl{T}^2=\{\{1,2,4\},\{1,2,5\},\{1,3,5\},\{2,3,6\},\{4,5,6\}\}$. The map $\gf$ is not an isomorphism because in the first one $L_6$ is a line it infinity and in the second one $L_6$ is not a line at infinity.
\section{\bf{Characterization of a Line at Infinity from the Nomenclature of an Infinity Type Arrangement}}
In this section we characterize a line at infinity using the nomenclature.

The theorem is stated as follows.
\begin{theorem}
\label{theorem:LineAtInfinity}	
Let $\mcl{L}_n=\{L_1,L_2,\ldots,L_n\}$ be an infinity type line arrangement	with infinity permutation $\gp$ with nomenclature $\gp(1)^{a_1}\gp(2)^{a_2}\ldots \gp(n)^{a_n}$. Then for some $1\leq t< n, L_{\gp(t)}$ with symbol $\gp(t)^{+1}$ is a line at infinity to the arrangement $\mcl{L}_n$ if and only if the following conditions hold.
\begin{enumerate}[label=\emph{(\arabic*)}]
\item Let $t<n,a_t=+1$ and there exists $u,t< u\leq n$ with $a_u=+1$ and there is no $w,u<w\leq n$ such that $a_w=+1$ and $\gp(t)<\gp(u)$.
\begin{enumerate}[label=\emph{(\Alph*)}]
\item Here all the symbols after $\gp(t)$ with $+1$ superscript are more than $\gp(t)$ and they increase as we move to the right.
\item All the symbols which occur after $\gp(t)$ with $-1$ superscript are less than all the symbols which occur after $\gp(t)$ with $+1$ superscript.
\item All the symbols which occur before $\gp(t)$ are more than $\gp(t)$ and lie in between those symbols which occur after $\gp(t)$ with $-1$ superscript and those symbols which occur after $\gp(t)$ with $+1$ superscript.
\item All the symbols after $\gp(t)$ and before $\gp(u)$ with $-1$ superscript are more than $\gp(t)$ and decrease as we move to the right. 
\item All the symbols after $\gp(u)$ have $-1$ superscript and they can be more than or less than $\gp(t)$. Among them those symbols which are less than $\gp(t)$ increase as we move to the right.   
Among them, those symbols which are more than $\gp(t)$ decrease as we move to the right and are smaller than those symbols with $-1$ superscript which are in between $\gp(t)$ and $\gp(u)$.
\end{enumerate}
\item Let $t<n,a_t=+1$ and there does not exist $u,t< u\leq n$ with $a_u=+1$.
\begin{enumerate}[label=\emph{(\roman*)}]
\item All the symbols after $\gp(t)$ have $-1$ superscript and they can be more than or less than $\gp(t)$. Among them those symbols which are less than $\gp(t)$ increase as we move to the right.   
Among them, those symbols which are more than $\gp(t)$ decrease as we move to the right.
\item All the symbols which occur before $\gp(t)$ and which are more than $\gp(t)$ are greater than all the symbols which occur after $\gp(t)$.
\item All the symbols which occur before $\gp(t)$ and which are less than $\gp(t)$ are lesser than all the symbols which occur after $\gp(t)$.
\end{enumerate}
\item Let $t<n,a_t=+1$ and there exists $u,t< u\leq n$ with $a_u=+1$ and there is no $w,u<w\leq n$ such that $a_w=+1$ and $\gp(u)<\gp(t)$.
\begin{enumerate}[label=\emph{(\alph*)}]
\item Here all the symbols after $\gp(t)$ with $+1$ superscript are less than $\gp(t)$ and they decrease as we move to the right.
\item All the symbols which occur after $\gp(t)$ with $-1$ superscript are more than all the symbols which occur after $\gp(t)$ with $+1$ superscript.
\item All the symbols which occur before $\gp(t)$ are less than $\gp(t)$ and lie in between those symbols which occur after $\gp(t)$ with $+1$ superscript and those symbols which occur after $\gp(t)$ with $-1$ superscript.
\item All the symbols after $\gp(t)$ and before $\gp(u)$ with $-1$ superscript are less than $\gp(t)$ and increase as we move to the right. 
\item All the symbols after $\gp(u)$ have $-1$ superscript and they can be more than or less than $\gp(t)$. Among them those symbols which are less than $\gp(t)$ increase as we move to the right.   
Among them, those symbols which are more than $\gp(t)$ decrease as we move to the right and are bigger than those symbols with $-1$ superscript which are in between $\gp(t)$ and $\gp(u)$.
\end{enumerate}
\end{enumerate}
\end{theorem}
We prove this theorem after the following two notes and an example.
\begin{note}
In the nomenclature $\gp(1)^{a_1}\gp(2)^{a_2}\ldots \gp(t)^{+1}\ldots\gp(n)^{a_n}$ the symbols after $\gp(t)$ with $-1$ superscript approach $\gp(t)$ that is those which are more than $\gp(t)$ decrease and those which are less than $\gp(t)$ will increase as we move to the right and the symbols after $\gp(t)$ with $+1$ superscript go far from $\gp(t)$ as we move to the right if $L_{\gp(t)}$ is a line at infinity.
\end{note}
\begin{note}
Let $\mcl{L}_n=\{L_1,L_2,\ldots,L_n\}$ be an infinity type line arrangement	with infinity permutation $\gp$ with nomenclature $\gp(1)^{a_1}\gp(2)^{a_2}\ldots \gp(n)^{a_n}$. Then for some $1\leq t< n, L_{\gp(t)}$ with symbol $\gp(t)^{-1}$ is a line at infinity to the arrangement $\mcl{L}_n$ if and only if
$L_{\gp(t)}$ with symbol $\gp(t)^{+1}$ is a line at infinity to the arrangement $\ti{\mcl{L}}_n=\{\ti{L}_1,\ti{L}_2,\ldots,\ti{L}_n\}$ with nomenclature $\gp(1)^{-a_1}\gp(2)^{-a_2}\ldots \gp(n)^{-a_n}$. Now from this we can infer the inequalities of symbols occurring in the nomenclature using Theorem~\emph{\ref{theorem:LineAtInfinity}}.
\end{note}
We illustrate this theorem via some examples before actually proving it.
\begin{example}
Suppose the nomenclature is $\gp(1)^{a_1}\gp(2)^{a_2}\gp(3)^{+1}\gp(4)^{+1}\gp(5)^{+1} \gp(6)^{-1}$  $\gp(7)^{-1}\gp(8)^{+1}\gp(9)^{+1}\gp(10)^{-1}\gp(11)^{-1}\gp(12)^{-1}\gp(13)^{-1}$ with $\gp(3)<\gp(9)$, $\gp(3)<\gp(10),\gp(3)<\gp(12),\gp(11)<\gp(3),\gp(13)<\gp(3)$.
Then using Theorem~\ref{theorem:LineAtInfinity}, the line $L_{\gp(3)}$ is a line at infinity if and only if we have 
\equa{&\gp(11)<\gp(13)<\gp(3)<\gp(12)<\gp(10)<\gp(7)<\gp(6)<\min\{\gp(1),\gp(2)\}<\\
	&\max\{\gp(1),\gp(2)\}<\gp(4)<\gp(5)<\gp(8)<\gp(9).}
We get $\gp(11)<\gp(13)<\gp(3)<\gp(12)<\gp(10)<\gp(7)<\gp(6)$ using condition \emph{(1):(D),(E)}. We get 
$\gp(4)<\gp(5)<\gp(8)<\gp(9)$ using condition \emph{(1):(A)} and we get $\gp(6)<\min\{\gp(1),\gp(2)\}<\max\{\gp(1),\gp(2)\}<\gp(4)$ using condition \emph{(1):(B),(C)}.
Here we have $t=3,n=13,u=9$. Since $\gp$ is a permutation we have $\gp(9)=13,\gp(8)=12,\gp(5)=11,\gp(4)=10,\{\gp(1),\gp(2)\}=\{8,9\},\gp(6)=7,\gp(7)=6,\gp(10)=5,\gp(12)=4,\gp(3)=3,\gp(13)=2,\gp(11)=1$. Hence $\gp=(1,8,12,4,10,5,11)(2,9,13)$ or $\gp=(1,9,13,2,8,12,4,10,5,11)$.
The corresponding line arrangements are respectively  $8^{a_1}9^{a_2}3^{+1}{10}^{+1}{11}^{+}7^{-1}6^{-1}{12}^{+1}$ ${13}^{+1}5^{-1}1^{-1}4^{-1}2^{-1}$ and $9^{a_1}8^{a_2}3^{+1}{10}^{+1}{11}^{+}7^{-1}6^{-1}{12}^{+1}{13}^{+1}5^{-1}1^{-1}4^{-1}2^{-1}$.

Suppose the nomenclature is $\gp(1)^{a_1}\gp(2)^{a_2}\gp(3)^{+1}\gp(4)^{+1}\gp(5)^{+1}\gp(6)^{-1}\gp(7)^{-1}\gp(8)^{+1}$  $\gp(9)^{+1}\gp(10)^{-1}\gp(11)^{-1}\gp(12)^{-1}\gp(13)^{-1}$ with $\gp(9)<\gp(3),\gp(3)<\gp(10)$, $\gp(3)<\gp(12),\gp(11)<\gp(3),\gp(13)<\gp(3)$.
Then the line $L_{\gp(3)}$ is a line at infinity if and only if we have 
\equa{&\gp(9)<\gp(8)<\gp(5)<\gp(4)<\min\{\gp(1),\gp(2)\}<\max\{\gp(1),\gp(2)\}<\gp(6)<\\
& \gp(7)<\gp(11)<\gp(13)<\gp(3)<\gp(12)<\gp(10).} So we have $\gp(10)=13,\gp(12)=12,\gp(3)=11,\gp(13)=10,\gp(11)=9,\gp(7)=8,\gp(6)=7,\{\gp(1),\gp(2)\}=\{5,6\},\gp(4)=4,\gp(5)=3,\gp(8)=2,\gp(9)=1$. Hence $\gp=(1,5,3,11,9)(2,6,7,8)(10,13)$ or $\gp=(1,6,7,8,2,5,3,11,9)(10,13)$. The corresponding line arrangements are respectively 
$5^{a_1}6^{a_2}11^{+1}4^{+1}3^{+1}7^{-1}8^{-1}$ $2^{+1}1^{+1}13^{-1}9^{-1}$ $12^{-1}10^{-1}$ and $6^{a_1}5^{a_2}11^{+1}4^{+1}3^{+1}7^{-1}8^{-1}2^{+1}1^{+1}13^{-1}9^{-1}12^{-1}10^{-1}$.

Theorem~\emph{\ref{theorem:LineAtInfinity}} can be used to conclude that a certain line is not a line at infinity as follows. Consider the line arrangement $1^{+1}2^{-1}5^{+1}3^{+1}4^{-1}6^{+1}$ in Figure~\emph{\ref{fig:Four}}. The line $L_4$ is not a line at infinity for this arrangement. We conclude this as follows. Corresponding to this line arrangement we consider another one given by $1^{-1}2^{+1}5^{-1}3^{-1}4^{+1}6^{-1}$. In this new arrangement after symbol $4$ only the symbol $6$ appears with $-1$ superscript. Now we want that all the symbols to the left of $4$ which are more than $4$ must be more than $6$ using condition \emph{(2):(ii)}. This is not true since $5$ occurs before $4$ and more than $4$ but not more than $6$. So $L_4$ is not a line at infinity.

Theorem~\emph{\ref{theorem:LineAtInfinity}} can be used to conclude that a certain line is a line at infinity as follows. For the arrangement $1^{+1}2^{-1}3^{+1}7^{+1}6^{+1}4^{-1}5^{+1}$ in Figure~\emph{\ref{fig:Zero}} we consider the corresponding arrangement $1^{-1}2^{+1}3^{-1}7^{-1}6^{-1}4^{+1}5^{-1}$. Now $L_4$ is a line at infinity because the symbols which occur before $4$ and which are more than $4$ are actually more than $6$ and those symbols which occur before $4$ and which are less than $4$ are actually less than $6$. Hence $L_4$ is a line at infinity for the arrangement in Figure~\emph{\ref{fig:Zero}}.
\end{example}
Now we prove Theorem~\ref{theorem:LineAtInfinity}.
\begin{proof}
If $t=n$ then $L_{\gp(n)}$ is a line at infinity to the line arrangement. So assume $t\neq n$, that is, $\gp(t)$ does not occur at the end of the nomenclature. 
Consider the following list of sub-symbols containing $\gp(t)$ that can occur in any nomenclature in this scenario when $t\neq n$.
\begin{enumerate}[label=(\arabic*)]
\item $\ldots \gp(t)^{b_t}\ldots \gp(s)^{b_s}\ldots \gp(u)^{+1}\ldots $ for $t<s<u$ and $\gp(t)<\gp(s)<\gp(u)$.
\item $\ldots \gp(s)^{b_s}\ldots \gp(t)^{b_t} \ldots \gp(u)^{+1}\ldots$ for $s<t<u$ and $\gp(t)<\gp(s)<\gp(u)$.
\item $\ldots \gp(t)^{b_t}\ldots \gp(s)^{b_s} \ldots \gp(u)^{-1}\ldots$ for $t<s<u$ and $\gp(t)<\gp(u)<\gp(s)$.
\item $\ldots \gp(s)^{b_s}\ldots \gp(t)^{b_t} \ldots \gp(u)^{-1}\ldots$ for $s<t<u$ and $\gp(t)<\gp(u)<\gp(s)$.
\item $\ldots \gp(t)^{b_t}\ldots \gp(s)^{b_s} \ldots \gp(u)^{-1}\ldots$ for $t<s<u$ and $\gp(u)<\gp(t)<\gp(s)$.
\item $\ldots \gp(s)^{b_s}\ldots \gp(t)^{b_t} \ldots \gp(u)^{-1}\ldots$ for $s<t<u$ and $\gp(u)<\gp(t)<\gp(s)$.
\item $\ldots \gp(t)^{b_t}\ldots \gp(s)^{b_s} \ldots \gp(u)^{-1}\ldots$ for $t<s<u$ and $\gp(s)<\gp(t)<\gp(u)$.
\item $\ldots \gp(s)^{b_s}\ldots \gp(t)^{b_t} \ldots \gp(u)^{-1}\ldots$ for $s<t<u$ and $\gp(s)<\gp(t)<\gp(u)$.
\item $\ldots \gp(t)^{b_t}\ldots \gp(s)^{b_s} \ldots \gp(u)^{+1}\ldots$ for $t<s<u$ and $\gp(u)<\gp(s)<\gp(t)$.
\item $\ldots \gp(s)^{b_s}\ldots \gp(t)^{b_t} \ldots \gp(u)^{+1}\ldots$ for $s<t<u$ and $\gp(u)<\gp(s)<\gp(t)$.
\item $\ldots \gp(t)^{b_t}\ldots \gp(s)^{b_s} \ldots \gp(u)^{-1}\ldots$ for $t<s<u$ and $\gp(s)<\gp(u)<\gp(t)$.
\item $\ldots \gp(s)^{b_s}\ldots \gp(t)^{b_t} \ldots \gp(u)^{-1}\ldots$ for $s<t<u$ and $\gp(s)<\gp(u)<\gp(t)$.
\end{enumerate}
In all the above $(1)-(12)$ cases the line $L_{\gp(t)}$ does not separate the origin and the vertex $L_{\gp(s)}\cap L_{\gp(u)}$.
\begin{enumerate}[label=(\arabic*$'$)]
	\item $\ldots \gp(t)^{b_t}\ldots \gp(s)^{b_s}\ldots \gp(u)^{-1}\ldots $ for $t<s<u$ and $\gp(t)<\gp(s)<\gp(u)$.
	\item $\ldots \gp(s)^{b_s}\ldots \gp(t)^{b_t} \ldots \gp(u)^{-1}\ldots$ for $s<t<u$ and $\gp(t)<\gp(s)<\gp(u)$.
	\item $\ldots \gp(t)^{b_t}\ldots \gp(s)^{b_s} \ldots \gp(u)^{+1}\ldots$ for $t<s<u$ and $\gp(t)<\gp(u)<\gp(s)$.
	\item $\ldots \gp(s)^{b_s}\ldots \gp(t)^{b_t} \ldots \gp(u)^{+1}\ldots$ for $s<t<u$ and $\gp(t)<\gp(u)<\gp(s)$.
	\item $\ldots \gp(t)^{b_t}\ldots \gp(s)^{b_s} \ldots \gp(u)^{+1}\ldots$ for $t<s<u$ and $\gp(u)<\gp(t)<\gp(s)$.
	\item $\ldots \gp(s)^{b_s}\ldots \gp(t)^{b_t} \ldots \gp(u)^{+1}\ldots$ for $s<t<u$ and $\gp(u)<\gp(t)<\gp(s)$.
	\item $\ldots \gp(t)^{b_t}\ldots \gp(s)^{b_s} \ldots \gp(u)^{+1}\ldots$ for $t<s<u$ and $\gp(s)<\gp(t)<\gp(u)$.
	\item $\ldots \gp(s)^{b_s}\ldots \gp(t)^{b_t} \ldots \gp(u)^{+1}\ldots$ for $s<t<u$ and $\gp(s)<\gp(t)<\gp(u)$.
	\item $\ldots \gp(t)^{b_t}\ldots \gp(s)^{b_s} \ldots \gp(u)^{-1}\ldots$ for $t<s<u$ and $\gp(u)<\gp(s)<\gp(t)$.
	\item $\ldots \gp(s)^{b_s}\ldots \gp(t)^{b_t} \ldots \gp(u)^{-1}\ldots$ for $s<t<u$ and $\gp(u)<\gp(s)<\gp(t)$.
	\item $\ldots \gp(t)^{b_t}\ldots \gp(s)^{b_s} \ldots \gp(u)^{+1}\ldots$ for $t<s<u$ and $\gp(s)<\gp(u)<\gp(t)$.
	\item $\ldots \gp(s)^{b_s}\ldots \gp(t)^{b_t} \ldots \gp(u)^{+1}\ldots$ for $s<t<u$ and $\gp(s)<\gp(u)<\gp(t)$.
\end{enumerate}
In all the above $(1')-(12')$ cases the line $L_{\gp(t)}$ does separate the origin and the vertex $L_{\gp(s)}\cap L_{\gp(u)}$.

If $L_{\gp(t)}$ is a line at infinity for the arrangement $\mcl{L}_n$ then possibilities in one set of  twelve cases occur but not in both sets. Now suppose possibilities in the first set occur, that is, $L_{\gp(t)}$ does not separate origin and the vertices of intersection of the arrangement. We have $b_t=1$.

\begin{enumerate}[label=(\alph*)]
\item Suppose we have the following list of sub-symbols.
\equ{\ldots \gp(p)^{b_p} \ldots \gp(q)^{b_q}\ldots \gp(t)^{+1} \ldots \gp(s)^{+1} \ldots \gp(u)^{+1} \ldots}
\begin{enumerate}[label=(\roman*)]
\item The sub-symbols $\ldots \gp(t)^{+1} \ldots \gp(s)^{+1} \ldots \gp(u)^{+1} \ldots$ implies that we have from cases $(1)$ and $(9)$ either $\gp(t)<\gp(s)<\gp(u)$ or $\gp(u)<\gp(s)<\gp(t)$.
\item The sub-symbols $\ldots \gp(q)^{b_q}\ldots \gp(t)^{+1} \ldots \gp(u)^{+1} \ldots $ implies that we have from cases $(2)$ and $(10)$ either $\gp(t)<\gp(q)<\gp(u)$ or $\gp(u)<\gp(q)<\gp(t)$.
\item The sub-symbols $\ldots \gp(p)^{b_p}\ldots \gp(t)^{+1} \ldots \gp(u)^{+1} \ldots $ implies that we have from cases $(2)$ and $(10)$ either $\gp(t)<\gp(p)<\gp(u)$ or $\gp(u)<\gp(p)<\gp(t)$.
\item The sub-symbols $\ldots \gp(q)^{b_q}\ldots \gp(t)^{+1} \ldots \gp(s)^{+1} \ldots $ implies that we have from cases $(2)$ and $(10)$ either $\gp(t)<\gp(q)<\gp(s)$ or $\gp(s)<\gp(q)<\gp(t)$.
\item The sub-symbols $\ldots \gp(p)^{b_p}\ldots \gp(t)^{+1} \ldots \gp(s)^{+1} \ldots $ implies that we have from cases $(2)$ and $(10)$ either $\gp(t)<\gp(p)<\gp(s)$ or $\gp(s)<\gp(p)<\gp(t)$.
\end{enumerate}
So we conclude from (a):(i)$-$(v) that if either $\gp(t)<\gp(s)$ or $\gp(t)<\gp(u)$ then we have 
\equ{\gp(t)<\min\{\gp(p),\gp(q)\}<\max\{\gp(p),\gp(q)\}<\gp(s)<\gp(u).}
We also conclude that if either $\gp(s)<\gp(t)$ or $\gp(u)<\gp(t)$ then we have  
\equ{\gp(u)<\gp(s)<\min\{\gp(p),\gp(q)\}<\max\{\gp(p),\gp(q)\}<\gp(t).}
\item Suppose we have the following list of sub-symbols.
\equ{\ldots \gp(p)^{b_p} \ldots \gp(q)^{b_q}\ldots \gp(t)^{+1} \ldots \gp(s)^{+1} \ldots \gp(u)^{-1} \ldots}
\begin{enumerate}[label=(\roman*)]
\item The sub-symbols $\ldots \gp(t)^{+1} \ldots \gp(s)^{+1} \ldots \gp(u)^{-1} \ldots$ implies that we have from cases $(3),(5),(7),(11)$ we have either $\gp(t)<\gp(u)<\gp(s)$ or $\gp(u)<\gp(t)<\gp(s)$ or $\gp(s)<\gp(t)<\gp(u)$ or $\gp(s)<\gp(u)<\gp(t)$.
\item The sub-symbols $\ldots \gp(q)^{b_q} \ldots \gp(t)^{+1} \ldots \gp(s)^{+1} \ldots$ implies that we have from cases $(2)$ and $(10)$ either $\gp(t)<\gp(q)<\gp(s)$ or $\gp(s)<\gp(q)<\gp(t)$. 
\item The sub-symbols $\ldots \gp(p)^{b_p} \ldots \gp(t)^{+1} \ldots \gp(s)^{+1} \ldots$ implies that we have from cases $(2)$ and $(10)$ either $\gp(t)<\gp(p)<\gp(s)$ or $\gp(s)<\gp(p)<\gp(t)$. 
\item The sub-symbols $\ldots \gp(q)^{b_q} \ldots \gp(t)^{+1} \ldots \gp(u)^{-1} \ldots$ implies that we have from cases $(4),(6),(8),(12)$ either $\gp(t)<\gp(u)<\gp(q)$ or $\gp(u)<\gp(t)<\gp(q)$ or $\gp(q)<\gp(t)<\gp(u)$ or $\gp(q)<\gp(u)<\gp(t)$.
\item The sub-symbols $\ldots \gp(p)^{b_p} \ldots \gp(t)^{+1} \ldots \gp(u)^{-1} \ldots$ implies that we have from cases $(4),(6),(8),(12)$ either $\gp(t)<\gp(u)<\gp(p)$ or $\gp(u)<\gp(t)<\gp(p)$ or $\gp(p)<\gp(t)<\gp(u)$ or $\gp(p)<\gp(u)<\gp(t)$.
\end{enumerate}
So we conclude from (b):(i)$-$(v) that if $\gp(t)<\gp(s)$ then
\equa{\text{either }&\gp(t)<\gp(u)<\min\{\gp(p),\gp(q)\}<\max\{\gp(p),\gp(q)\}<\gp(s)\\
\text{or }&\gp(u)<\gp(t)<\min\{\gp(p),\gp(q)\}<\max\{\gp(p),\gp(q)\}<\gp(s).}
We also conclude that if $\gp(s)<\gp(t)$ then 
\equa{\text{either }&\gp(s)<\min\{\gp(p),\gp(q)\}<\max\{\gp(p),\gp(q)\}<\gp(u)<\gp(t)\\
\text{or }&\gp(s)<\min\{\gp(p),\gp(q)\}<\max\{\gp(p),\gp(q)\}<\gp(t)<\gp(u).}
\item Suppose we have the following list of sub-symbols.
\equ{\ldots \gp(p)^{b_p} \ldots \gp(q)^{b_q}\ldots \gp(t)^{+1} \ldots \gp(s)^{-1} \ldots \gp(u)^{+1} \ldots}	
\begin{enumerate}[label=(\roman*)]
\item The sub-symbols $\ldots \gp(t)^{+1} \ldots \gp(s)^{-1} \ldots \gp(u)^{+1} \ldots$ implies that we have from cases $(1)$ and $(9)$ we have either $\gp(t)<\gp(s)<\gp(u)$ or $\gp(u)<\gp(s)<\gp(t)$.
\item The sub-symbols $\ldots \gp(q)^{b_q} \ldots \gp(t)^{+1} \ldots \gp(u)^{+1} \ldots$ implies that we have from cases $(2)$ and $(10)$ we have either $\gp(t)<\gp(q)<\gp(u)$ or $\gp(u)<\gp(q)<\gp(t)$.
\item The sub-symbols $\ldots \gp(p)^{b_p} \ldots \gp(t)^{+1} \ldots \gp(u)^{+1} \ldots$ implies that we have from cases $(2)$ and $(10)$ we have either $\gp(t)<\gp(p)<\gp(u)$ or $\gp(u)<\gp(p)<\gp(t)$.
\item The sub-symbols $\ldots \gp(q)^{b_q} \ldots \gp(t)^{+1} \ldots \gp(s)^{-1} \ldots$ implies that we have from cases $(4),(6),(8),(12)$ either $\gp(t)<\gp(s)<\gp(q)$ or $\gp(s)<\gp(t)<\gp(q)$ or $\gp(q)<\gp(t)<\gp(s)$ or $\gp(q)<\gp(s)<\gp(t)$.
\item The sub-symbols $\ldots \gp(p)^{b_p} \ldots \gp(t)^{+1} \ldots \gp(s)^{-1} \ldots$ implies that we have from cases $(4),(6),(8),(12)$ either $\gp(t)<\gp(s)<\gp(p)$ or $\gp(s)<\gp(t)<\gp(p)$ or $\gp(p)<\gp(t)<\gp(s)$ or $\gp(p)<\gp(s)<\gp(t)$.
\end{enumerate}
So we conclude from (c):(i)$-$(v) that if $\gp(t)<\gp(u)$ then 
\equ{\gp(t)<\gp(s)<\min\{\gp(p),\gp(q)\}<\max\{\gp(p),\gp(q)\}<\gp(u).}
We also conclude that if $\gp(u)<\gp(t)$ then 
\equ{\gp(u)<\min\{\gp(p),\gp(q)\}<\max\{\gp(p),\gp(q)\}<\gp(s)<\gp(t).}
\item Suppose we have the following list of sub-symbols.
\equ{\ldots \gp(p)^{b_p} \ldots \gp(q)^{b_q}\ldots \gp(t)^{+1} \ldots \gp(s)^{-1} \ldots \gp(u)^{-1} \ldots}	
\begin{enumerate}[label=(\roman*)]
\item The sub-symbols $\ldots \gp(t)^{+1} \ldots \gp(s)^{-1} \ldots \gp(u)^{-1} \ldots$ implies that we have from cases $(3),(5),(7),(11)$ we have either $\gp(t)<\gp(u)<\gp(s)$ or $\gp(u)<\gp(t)<\gp(s)$ or $\gp(s)<\gp(t)<\gp(u)$ or $\gp(s)<\gp(u)<\gp(t)$.
\item The sub-symbols $\ldots \gp(q)^{b_q} \ldots \gp(t)^{+1} \ldots \gp(u)^{-1} \ldots$ implies that we have from cases $(4),(6),(8),(12)$ either $\gp(t)<\gp(u)<\gp(q)$ or $\gp(u)<\gp(t)<\gp(q)$ or $\gp(q)<\gp(t)<\gp(u)$ or $\gp(q)<\gp(u)<\gp(t)$.
\item The sub-symbols $\ldots \gp(p)^{b_q} \ldots \gp(t)^{+1} \ldots \gp(u)^{-1} \ldots$ implies that we have from cases $(4),(6),(8),(12)$ either $\gp(t)<\gp(u)<\gp(p)$ or $\gp(u)<\gp(t)<\gp(p)$ or $\gp(p)<\gp(t)<\gp(u)$ or $\gp(p)<\gp(u)<\gp(t)$.
\item The sub-symbols $\ldots \gp(q)^{b_q} \ldots \gp(t)^{+1} \ldots \gp(s)^{-1} \ldots$ implies that we have from cases $(4),(6),(8),(12)$ either $\gp(t)<\gp(s)<\gp(q)$ or $\gp(s)<\gp(t)<\gp(q)$ or $\gp(q)<\gp(t)<\gp(s)$ or $\gp(q)<\gp(s)<\gp(t)$.
\item The sub-symbols $\ldots \gp(p)^{b_q} \ldots \gp(t)^{+1} \ldots \gp(s)^{-1} \ldots$ implies that we have from cases $(4),(6),(8),(12)$ either $\gp(t)<\gp(s)<\gp(p)$ or $\gp(s)<\gp(t)<\gp(p)$ or $\gp(p)<\gp(t)<\gp(s)$ or $\gp(p)<\gp(s)<\gp(t)$.
\end{enumerate}
So we conclude from (d):(i)$-$(v) that if $\gp(t)<\gp(u)<\gp(s)$ then
\equa{\text{either }&\min\{\gp(p),\gp(q)\}<\max\{\gp(p),\gp(q)\}<\gp(t)<\gp(u)<\gp(s)\\
\text{or }&\min\{\gp(p),\gp(q)\}<\gp(t)<\gp(u)<\gp(s)<\max\{\gp(p),\gp(q)\}\\
\text{or }&\gp(t)<\gp(u)<\gp(s)<\min\{\gp(p),\gp(q)\}<\max\{\gp(p),\gp(q)\}.}  	
We conclude that if $\gp(u)<\gp(t)<\gp(s)$ then 
\equa{\text{either }&\min\{\gp(p),\gp(q)\}<\max\{\gp(p),\gp(q)\}<\gp(u)<\gp(t)<\gp(s)\\
	\text{or }&\min\{\gp(p),\gp(q)\}<\gp(u)<\gp(t)<\gp(s)<\max\{\gp(p),\gp(q)\}\\
	\text{or }&\gp(u)<\gp(t)<\gp(s)<\min\{\gp(p),\gp(q)\}<\max\{\gp(p),\gp(q)\}.}  	
We conclude that if $\gp(s)<\gp(t)<\gp(u)$ then 
\equa{\text{either }&\min\{\gp(p),\gp(q)\}<\max\{\gp(p),\gp(q)\}<\gp(s)<\gp(t)<\gp(u)\\
	\text{or }&\min\{\gp(p),\gp(q)\}<\gp(s)<\gp(t)<\gp(u)<\max\{\gp(p),\gp(q)\}\\
	\text{or }&\gp(s)<\gp(t)<\gp(u)<\min\{\gp(p),\gp(q)\}<\max\{\gp(p),\gp(q)\}.}  	
We conclude that if $\gp(s)<\gp(u)<\gp(t)$ then
\equa{\text{either }&\min\{\gp(p),\gp(q)\}<\max\{\gp(p),\gp(q)\}<\gp(s)<\gp(u)<\gp(t)\\
	\text{or }&\min\{\gp(p),\gp(q)\}<\gp(s)<\gp(u)<\gp(t)<\max\{\gp(p),\gp(q)\}\\
	\text{or }&\gp(s)<\gp(u)<\gp(t)<\min\{\gp(p),\gp(q)\}<\max\{\gp(p),\gp(q)\}.}  	
\end{enumerate}
From these inequalities we infer the conditions (1):(A)$-$(E),(2):(i)$-$(iii), 
(3):(a)$-$(e) of the theorem. The converse also holds. This proves the theorem.
\end{proof}

\section{\bf{An Open Question}}
In this section we pose an open question for an arbitrary line arrangement which need not be of infinity type and need not have global cyclicity.
\begin{ques}
Let $\mcl{L}_n=\{L_1,L_2,\ldots,L_n\}$ be a line arrangement in the plane. Give a combinatorial nomenclature for the line arrangement and describe the triangles present in the line arrangement combinatorially.
\end{ques}
\section{\bf{Mathematica Function Code which outputs the List of Triangles of an Infinity Type Line Arrangement}}
\label{sec:Mathematica}
Given below, is the mathematica code for a function, whose input is the nomenclature of an infinity type line arrangement and whose output is the list of triangles in the arrangement.


\section*{Supplementary material}

(* The function "GetSign" is an auxiliary function used in the main function "LasTriangles" *)

```
GetSign[u_, v_]:=Module[{},
If[(((u[[1]] < u[[2]])&&(u[[2]] < u[[3]])&&(v[[1]] == -1))||
((u[[2]] < u[[3]])&&(u[[3]] < u[[1]])&&(v[[1]] == -1))||
((u[[3]] < u[[1]])&&(u[[1]] < u[[2]])&&(v[[1]] == 1))||
((u[[1]] < u[[3]])&&(u[[3]] < u[[2]])&&(v[[1]] == -1))||
((u[[3]] < u[[2]])&&(u[[2]] < u[[1]])&&(v[[1]] == -1))||
((u[[2]] < u[[1]])&&(u[[1]] < u[[3]])&&(v[[1]] == 1)), Return[1], Return[-1]];
];
```

(* The function "LasTriangles" gives those triangles of the line arrangement with nomenclature $u[[1]]^{a[[1]]} u[[2]]^{a[[2]]} \cdots u[[n]]^{a[[n]]}$ which has the line u[[n]] as one of the sides *)

```
LasTriangles[u_, a_]:=Module[{Tri = {}, i = 0, Boo = False, v = {}, w = {}, f = {}, s = {}, j = 0,
n = Length[u], b = {}, d = {}},
v = Table[a[[i]], {i, 3, n}];
w = Sort[Table[u[[i]], {i, 1, n - 1}]];
b = Flatten[Position[u, w[[1]]]];
d = Flatten[Position[u, w[[n - 1]]]];
For[i = 1, i ≤ n - 2, i++,
f = Flatten[Position[u, w[[i]]]];
s = Flatten[Position[u, w[[i + 1]]]];

If[(((f[[1]]==1)&&(s[[1]] == 2)&&(GetSign[u, v] == 1))||
((f[[1]] == 2)&&(s[[1]] == 1)&&(GetSign[u, v] == -1)),
If[((u[[1]] < u[[f[[1]]]])&&(v[[n - 2]] == -1))||((u[[n]] > u[[s[[1]]]])&&(v[[n - 2]] == 1)),
Boo = True, Boo = False];
If[Boo == True,
j = 3;
While[j ≤ n - 1,
If[((u[[j]] < u[[f[[1]]]])&&(v[[j - 2]] == 1))||((u[[j]] > u[[s[[1]]]])&&(v[[j - 2]] == -1)),
Boo = True, Boo = False;
];
If[Boo == False, j = n];
j++;
];];];
If[Boo == True, Tri = Append[Tri, Sort[{u[[1]], u[[2]], u[[n]]}]]];
Boo = False;


If[(((f[[1]]==1)&&(s[[1]] == 2)&&(GetSign[u, v] == -1))||
((f[[1]] == 2)&&(s[[1]] == 1)&&(GetSign[u, v] == 1)),
If[((u[[n]] < u[[f[[1]]]])&&(v[[n - 2]] == 1))||((u[[n]] > u[[s[[1]]]])&&(v[[n - 2]] == -1)),
Boo = True, Boo = False];
If[Boo == True,
```

$j = 3$;
While[$j \leq n - 1$,
If[$((u[[j]] < u[[f[[1]]]])\&\&(v[[j-2]] == -1))\|((u[[j]] > u[[s[[1]]]])\&\&(v[[j-2]] == 1))$,
Boo = True, Boo = False;
];
If[Boo == False, $j = n$];
$j$++;
];];];
If[Boo == True, Tri = Append[Tri, Sort[$\{u[[1]], u[[2]], u[[n]]\}$]]];
Boo = False;


If[$(f[[1]] > 2)\|(s[[1]] > 2)$,
If[$((f[[1]] < s[[1]])\&\&(v[[s[[1]] - 2]] == 1))\|((s[[1]] < f[[1]])\&\&(v[[f[[1]] - 2]] == -1))$,
If[$((u[[n]] < u[[f[[1]]]])\&\&(v[[n-2]] == -1))\|((u[[n]] > u[[s[[1]]]])\&\&(v[[n-2]] == 1))$,
Boo = True, Boo = False];
If[Boo == True,
$j = 1 + \text{Max}[f[[1]], s[[1]]]$;
While[$j \leq n - 1$,
If[$((u[[j]] < u[[f[[1]]]])\&\&(v[[j-2]] == 1))\|((u[[j]] > u[[s[[1]]]])\&\&(v[[j-2]] == -1))$,
Boo = True, Boo = False;
];
If[Boo == False, $j = n$];
$j$++;
];];];];
If[Boo == True, Tri = Append[Tri, Sort[$\{u[[f[[1]]]], u[[s[[1]]]], u[[n]]\}$]]];
Boo = False;


If[$(f[[1]] > 2)\|(s[[1]] > 2)$,
If[$((f[[1]] < s[[1]])\&\&(v[[s[[1]] - 2]] == -1))\|((s[[1]] < f[[1]])\&\&(v[[f[[1]] - 2]] == 1))$,
If[$((u[[n]] < u[[f[[1]]]])\&\&(v[[n-2]] == 1))\|((u[[n]] > u[[s[[1]]]])\&\&(v[[n-2]] == -1))$,
Boo = True, Boo = False];
If[Boo == True,
$j = 1 + \text{Max}[f[[1]], s[[1]]]$;
While[$j \leq n - 1$,
If[$((u[[j]] < u[[f[[1]]]])\&\&(v[[j-2]] == -1))\|((u[[j]] > u[[s[[1]]]])\&\&(v[[j-2]] == 1))$,
Boo = True, Boo = False;
];
If[Boo == False, $j = n$];
$j$++;
];];];];
If[Boo == True, Tri = Append[Tri, Sort[$\{u[[f[[1]]]], u[[s[[1]]]], u[[n]]\}$]]];
Boo = False;
];
Boo = False;


If[$((b[[1]] == 1)\&\&(d[[1]] == 2)\&\&(\text{GetSign}[u, v] == 1))\|$

```
((b[[1]] == 2)&&(d[[1]] == 1)&&(GetSign[u, v] == 1)),
If((u[[n]] < u[[d[[1]]]])&&(u[[n]] > u[[b[[1]]]])&&(v[[n − 2]] == 1)), Boo = True, Boo = False];
If[Boo == True,
j = 3;
While[j ≤ n − 1,
If[v[[j − 2]] == −1, Boo = True, Boo = False];
If[Boo == False, j = n];
j++;
];];];
If[Boo == True, Tri = Append[Tri, Sort[{u[[b[[1]]]], u[[d[[1]]]], u[[n]]}]]];
Boo = False;


If((b[[1]] == 1)&&(d[[1]] == 2)&&(GetSign[u, v] == −1))||
((b[[1]] == 2)&&(d[[1]] == 1)&&(GetSign[u, v] == −1)),
If((u[[n]] < u[[d[[1]]]])&&(u[[n]] > u[[b[[1]]]])&&(v[[n − 2]] == −1)), Boo = True, Boo = False];
If[Boo == True,
j = 3;
While[j ≤ n − 1,
If[v[[j − 2]] == 1, Boo = True, Boo = False];
If[Boo == False, j = n];
j++;
];];];
If[Boo == True, Tri = Append[Tri, Sort[{u[[b[[1]]]], u[[d[[1]]]], u[[n]]}]]];
Boo = False;


If((b[[1]] > 2)||(d[[1]] > 2)),
If((b[[1]] < d[[1]])&&(v[[d[[1]] − 2]] == 1))||((d[[1]] < b[[1]])&&(v[[b[[1]] − 2]] == 1)),
If((u[[n]] < u[[d[[1]]]])&&(u[[n]] > u[[b[[1]]]])&&(v[[n − 2]] == 1)), Boo = True, Boo = False];
If[Boo == True,
j = 1 + Max[b[[1]], d[[1]]];
While[j ≤ n − 1,
If[(v[[j − 2]] == −1), Boo = True, Boo = False];
If[Boo == False, j = n];
j++;
];];];];
If[Boo == True, Tri = Append[Tri, Sort[{u[[b[[1]]]], u[[d[[1]]]], u[[n]]}]]];
Boo = False;


If((b[[1]] > 2)||(d[[1]] > 2)),
If((b[[1]] < d[[1]])&&(v[[d[[1]] − 2]] == −1))||((d[[1]] < b[[1]])&&(v[[b[[1]] − 2]] == −1)),
If((u[[n]] < u[[d[[1]]]])&&(u[[n]] > u[[b[[1]]]])&&(v[[n − 2]] == −1)), Boo = True, Boo = False];
If[Boo == True,
j = 1 + Max[b[[1]], d[[1]]];
While[j ≤ n − 1,
If[(v[[j − 2]] == 1), Boo = True, Boo = False];
If[Boo == False, j = n];
```

$j$++;
];];];];
If[Boo == True, Tri = Append[Tri, Sort[{$u[[b[[1]]]], u[[d[[1]]]], u[[n]]$}]]];
Return[Sort[Tri, (#1[[1]] < #2[[1]])||((#1[[1]] == #2[[1]])&&(#1[[2]] < #2[[2]]))||
((#1[[1]] == #2[[1]])&&(#1[[2]] == #2[[2]])&&(#1[[3]] < #2[[3]]))&]];
];

(* The function "Triangles" gives all the triangles present in the infinity type line arrangement with nomenclature $u[[1]]^{v[[1]]}u[[2]]^{v[[2]]} \cdots u[[n]]^{v[[n]]}$ *)

Triangles[u_, v_]:=Module[{$n$ = Length[$u$], Tri = {}, $i = 0, j = 0$},
Tri = {Sort[{$u[[1]], u[[2]], u[[3]]$}]};
For[$i = 4, i \leq n, i$++,
Tri = Join[Tri, LasTriangles[Table[$u[[j]], \{j, 1, i\}$], Table[$v[[j]], \{j, 3, i\}$]]]];
];
Return[Sort[Tri, (#1[[1]] < #2[[1]])||((#1[[1]] == #2[[1]])&&(#1[[2]] < #2[[2]]))||
((#1[[1]] == #2[[1]])&&(#1[[2]] == #2[[2]])&&(#1[[3]] < #2[[3]]))&]];
];

(* Some Examples Illustrating the Function "Triangles"*)

(* A Seven Line Arrangment with Two Different Nomenclatures :
$1^{+1}2^{-1}3^{+1}7^{+1}6^{+1}4^{-1}5^{+1}$and $1^{+1}2^{-1}3^{+1}4^{-1}7^{+1}6^{+1}5^{+1}$ *)

Triangles[{1, 2, 3, 7, 6, 4, 5}, {1, −1, 1, 1, 1, −1, 1}]
Triangles[{1, 2, 3, 4, 7, 6, 5}, {1, −1, 1, −1, 1, 1, 1}]
{{1, 2, 3}, {1, 2, 4}, {1, 6, 7}, {2, 3, 7}, {5, 6, 7}}
{{1, 2, 3}, {1, 2, 4}, {1, 6, 7}, {2, 3, 7}, {5, 6, 7}}

(* Two Different Six Line Arrangements with the Same Set of Triangles :
$1^{+1}2^{-1}5^{+1}3^{+1}4^{-1}6^{+1}$and $1^{+1}2^{-1}5^{+1}3^{+1}6^{+1}4^{-1}$ *)

Triangles[{1, 2, 5, 3, 4, 6}, {1, −1, 1, 1, −1, 1}]
Triangles[{1, 2, 5, 3, 6, 4}, {1, −1, 1, 1, 1, −1}]
{{1, 2, 4}, {1, 2, 5}, {1, 3, 5}, {2, 3, 6}, {4, 5, 6}}
{{1, 2, 4}, {1, 2, 5}, {1, 3, 5}, {2, 3, 6}, {4, 5, 6}}

Post Doctoral Fellow in Mathematics, Harish-Chandra Research Institute, Chhatnag Road, Jhunsi, Prayagraj (Allahabad)-211019, Uttar Pradesh, INDIA

*Email address*: `akcp1728@gmail.com`



\end{document}